\documentclass[bj,preprint]{imsart}

\RequirePackage{amsthm,amsmath,amsfonts,amssymb}
\RequirePackage[numbers]{natbib}

\RequirePackage[colorlinks,citecolor=blue,urlcolor=blue]{hyperref}
\RequirePackage{graphicx}
\RequirePackage{caption}
\RequirePackage{subcaption}
\RequirePackage{mathtools}
\RequirePackage[dvipsnames]{xcolor}

\startlocaldefs
\numberwithin{equation}{section}
\theoremstyle{plain}
\newtheorem{theorem}{Theorem}[section] 

\newtheorem{lemma}{Lemma}[section]
\newtheorem{prop}{Proposition}[section]
\newtheorem{conjecture}{Conjecture}[section]
\theoremstyle{remark}

\newcommand{\E}[1]{\ensuremath\mathbb{E}\left[#1\right]}
\newcommand{\ind}[1]{\ensuremath\mathbb{I} \left\{#1\right\}}
 
\newcommand{\R}{\ensuremath \mathbb{R}}
\newcommand{\N}{\mathcal{N}}
\newcommand{\convp}{\ensuremath \stackrel{\mathrm{P}}{\longrightarrow}}
\newcommand{\convP}{\stackrel{\textrm{P}}{\longrightarrow}}

\newcommand{\convd}{\ensuremath \stackrel{\mathrm{d}}{\longrightarrow}}

\newcommand{\bb}{\boldsymbol{\beta}}
\newcommand{\bx}{\boldsymbol{x}}
\newcommand{\gmle}{g_{\text{MLE}}}

\newcommand{\mbf}[1]{\boldsymbol{#1}}
\newcommand{\prox}{\mathsf{prox}}
\newcommand{\zt}{T}

\DeclareMathOperator{\llr}{LLR} 
\DeclareMathOperator{\prob}{\mathbb{P}}
\DeclareMathOperator{\var}{Var}

\endlocaldefs

\begin{document}

\captionsetup[figure]{labelfont=bf, labelformat=simple, labelsep=period}

\begin{frontmatter}
\title{The Asymptotic Distribution of the MLE in High-Dimensional
  Logistic Models: Arbitrary Covariance}
\runtitle{Logistic MLE: arbitrary covariance}

\begin{aug}
\author[A]{\fnms{Qian} \snm{Zhao}\ead[label=e1, mark]{qzhao1@stanford.edu}},
\author[B]{\fnms{Pragya} \snm{Sur}\ead[label=e2]{pragya@fas.harvard.edu}}
\and
\author[C]{\fnms{Emmanuel} \snm{J. Cand\`es}\ead[label=e3]{candes@stanford.edu}}

\address[A]{
Department of Statistics, 
Stanford University,
390 Jane Stanford Way,
Stanford, CA 94305-4020,
USA
\printead{e1}
}

\address[B]{
Department of Statistics,
Harvard University,
Science Center 712,
One Oxford Street,
Cambridge, MA 02138,
USA
\printead{e2}}

\address[C]{
Department of Mathematics and of Statistics, 
Stanford University,
Sequoia Hall, 
390 Jane Stanford Way,
Stanford, CA 94305-4020,
USA
\printead{e3}
}
\end{aug}

\begin{abstract}
 We study the distribution of the maximum likelihood estimate (MLE)
  in high-dimensional logistic models, where covariates are Gaussian with
  an arbitrary covariance structure. We prove that in the limit of
  large problems holding the ratio between the number $p$ of
  covariates and the sample size $n$ constant, every finite list of
  MLE coordinates follows a multivariate normal
  distribution. Concretely, the $j$th coordinate $\hat {\beta}_j$ of
  the MLE is asymptotically normally distributed with mean
  $\alpha_\star \beta_j$ and standard deviation $\sigma_\star/\tau_j$;
  here, $\beta_j$ is the value of the true regression coefficient, and
  $\tau_j$ the standard deviation of the $j$th predictor conditional
  on all the others. The numerical parameters $\alpha_\star > 1$ and
  $\sigma_\star$ only depend upon the problem dimensionality $p/n$ and
  the overall signal strength, and can be accurately estimated. Our
  results imply that the MLE's magnitude is biased upwards and that
  the MLE's standard deviation is greater than that predicted by
  classical theory. We present a series of experiments on simulated
  and real data showing excellent agreement with the theory.
\end{abstract}


\end{frontmatter}

\section{Introduction}
\label{sec:intro}

Logistic regression is the most widely applied statistical model for
fitting a binary response from a list of covariates. This model is
used in a great number of disciplines ranging from social science to
biomedical studies. For instance, logistic regression is routinely
used to understand the association between the susceptibility of a
disease and genetic and/or environmental risk factors.

A logistic model is usually constructed by the method of maximum
likelihood (ML) and it is therefore critically important to understand
the properties of ML estimators (MLE) in order to test hypotheses,
make predictions and understand their validity. In this regard,
assuming the logistic model holds, classical ML theory provides the
asymptotic distribution of the MLE when the number of observations $n$
tends to infinity while the {\em number $p$ of variables remains
  constant.} In a nutshell, the MLE is asymptotically normal with mean
equal to the true vector of regression coefficients and variance equal
to $\mathcal{I}^{-1}_{\mbf{\beta}}$, where $\mathcal{I}_{\mbf{\beta}} $
is the Fisher information evaluated at true coefficients
\cite[Appendix~A]{mccullagh1989}, \cite[Chapter~5]{vaart1998}. Another
staple of classical ML theory is that the extensively used likelihood
ratio test (LRT) asymptotically follows a chi-square distribution
under the null, a result known as Wilk's theorem
\cite{wilk1938}\cite[Chapter~16]{vaart1998}. Again, this holds in the
limit where $p$ is fixed and $n \rightarrow \infty$ so that the
dimensionality $p/n$ is vanishingly small. (See \cite{portnoy1984asymptotic,portnoy1985asymptotic,portnoy1988,he2000parameters,fan2019nonuniformity,anastasiou2020bounds} for the relevance of these classical results under diverging  dimensions with $p$ negligible compared to $n$.)

\subsection{High-dimensional maximum-likelihood theory}

Against this background, a recent paper \cite{sur18} showed that the
classical theory does not even approximately hold in large sample
sizes if $p$ is not negligible compared to $n$.  In more details,
empirical and theoretical analyses in \cite{sur18} establish the
following conclusions:
\begin{enumerate}
\item[1.] The MLE is biased in that it
overestimates the true effect
magnitudes. 
\item[2.] The variance of the MLE is larger than that implied
  by the inverse Fisher information.
\item[3.] The LRT is not distributed as a chi-square variable; it
  is stochastically larger than a chi-square.
\end{enumerate}
Under a suitable model for the covariates, \cite{sur18} developed
formulas to calculate the asymptotic bias and variance of the MLE
under a limit of large samples where the ratio $p/n$ between the
number of variables and the sample size has a positive limit
$\kappa$. Operationally, these results provide an excellent
approximation of the distribution of the MLE in large logistic models
in which the number of variables obey $p \approx \kappa \, n$ (this is
the same regime as that considered in random matrix theory when
researchers study the eigenvalue distribution of sample covariance
matrices in high dimensions).  Furthermore, \cite{sur18} also proved
that the LRT is asymptotically distributed as a fixed multiple of a
chi-square, with a multiplicative factor that can be determined.

\subsection{This paper}

The asymptotic distribution of the MLE in high-dimensional logistic
regression briefly reviewed above holds for models in which the
covariates are {\em independent} and Gaussian. This is the starting
point of this paper: since features typically encountered in
applications are not independent, it is important to describe the
behavior of the MLE under models with arbitrary covariance
structures. In this work, we shall limit ourselves to Gaussian
covariates although we believe our results extend to a wide class of
distributions with sufficiently light tails (we provide numerical evidence supporting this claim). 

To give a glimpse of our results, imagine we have $n$ independent
pairs of observations $(\boldsymbol{x}_i, y_i)$, where the features
$\bx_i \in \R^p$ and the class label $y_i \in \{-1,1\}$. We assume that
the $\bx_i$'s follow a multivariate normal distribution with mean zero
and arbitrary covariance, and that the likelihood of the class label
$y_i$ is related to $\bx_i$ through the logistic model 
\begin{equation}
  \label{eq:logistic}
  \prob(y_i=1\:|\:\boldsymbol{x}_i) =
  1/(1+e^{-\boldsymbol{x}_i^\top \boldsymbol{\beta}}).
\end{equation}
Denote the MLE for estimating the parameters $\mbf{\beta}$ by
$\hat{\mbf{\beta}}$ and consider centering and scaling
$\hat{\mbf{\beta}}$ via 
\begin{equation}
\label{eq:Z}
\zt_j = \frac{\sqrt{n}(\hat{\beta}_j-\alpha_\star\beta_j)}{\sigma_\star/\tau_j};
\end{equation}
here, $\tau_j$ is the standard deviation of the $j$th feature variable
(the $j$th component of $\bx_i$) conditional on all the other
variables (all the other components of $\bx_i$), whereas
$\alpha_\star > 1$ and $\sigma_\star$ are numerical parameters we
shall determine in Section \ref{sec:finitemle}. Then after
establishing a stochastic representation for the MLE which is valid
for every finite $n$ and $p$, this paper proves two distinct
asymptotic results (both hold in the same regime where $n$ and $p$
diverge to infinity in a fixed ratio).

The first concerns {\bf marginals of the MLE}. Under some conditions
on the magnitude of the regression coefficient $\beta_j$, we show
that\footnote{Throughout, $\convd$ (resp.~$\convP$) is a shorthand for
  convergence in distribution (resp.~probability).} 
 \begin{equation}\label{eq:Z_conv}
 \zt_j \,\, \convd \,\, \mathcal{N}(0,1),
\end{equation}
and demonstrate an analogous statement for the distribution of any
finite collection of coordinates. The meaning is clear; if the
statistician is given several data sets from the model above and
computes a given regression coefficient for each via ML, then the
histogram of these coefficients across all data sets will look
approximately Gaussian. 

This state of affairs extends \cite[Theorem~3]{sur18} significantly,
which established the joint distribution of a finite collection of
\emph{null} coordinates, in the setting of independent
covariates. Specifically,
\begin{enumerate}
\item  Eqn.~\eqref{eq:Z_conv} is true when covariates have arbitrary covariance structure.
\item Eqn.~\eqref{eq:Z_conv} holds for both null and non-null coordinates.
\end{enumerate}

The second asymptotic result concerns {\bf the empirical distribution
  of the MLE in a single data set/realization}: we prove that
the empirical distribution of the $\zt_j$'s
converges to a standard normal in the sense that, 
\begin{equation}\label{eq:intro-2}
  \frac{\# \{j: \zt_j \le t\}}{p} 
\convp \,\, \mathbb{P}(\N(0,1) \le t). 
\end{equation}

This means that if we were to plot the histogram of all the
$\zt_j$'s obtained from a single data set, we would just see a
bell curve. Another consequence is that for sufficiently nice
functions $f(\cdot)$, we have 
\begin{equation}\label{eq:weakconv}
\frac{1}{p} \sum_{j=1}^p f(\zt_j) \,\, \stackrel{\textrm{P}}{\longrightarrow}
\,\, \E{f(Z)}, 
\end{equation}
where $Z \sim \mathcal{N}(0,1)$. For instance, taking $f$ to be the absolute value---we use the
caveat that $f$ is not uniformly bounded---we would 
conclude that
\[
\frac{1}{p} \sum_{j=1}^p \sqrt{n} \, \tau_j |\hat{\beta}_j - \alpha_\star
\beta_j| \,\, \stackrel{\textrm{P}}{\longrightarrow} \,\,
\sigma_\star \, \sqrt{2/{\pi}}.
\]
Taking $f$ to be the indicator function of the interval $[-1.96,
1.96]$, we would see that 
\[
\frac{1}{p} \sum_{j=1}^p 1\left\{ -1.96 \le
\frac{\sqrt{n}(\hat{\beta}_j-\alpha_\star\beta_j)}{\sigma_\star/\tau_j}
\le 1.96\right\}  \,\, \stackrel{\textrm{P}}{\longrightarrow} \,\,
0.95.
\]
Hence, the miscoverage rate (averaged over all variables) of the
confidence intervals
\[
[\hat{\beta}_j^-, \hat{\beta}_j^+], \qquad \hat{\beta}_j^\pm = \frac{\hat{\beta}_j \pm 1.96 \sigma_\star/\sqrt{n}\tau_j}{\alpha_\star}, 
\]
in a single experiment would approximately be equal to 5\%.

Finally, this paper extends the LRT asymptotics to the case of
arbitrary covariance. 

We provide an R package ``glmhd'' available on GitHub \cite{glmhd},
which provides functionality to compute the parameters $\alpha_\star$,
$\sigma_\star$ discussed above and $\lambda_\star$ introduced below,
as well as functionality to analyze real datasets.  

\subsection{Technical contributions}
  
We derive the asymptotic distribution of every finite collection of
MLE coordinates (Theorem \ref{thm:finitemle}). In \cite{sur18}, the
authors studied the distribution of the MLE corresponding to a
\emph{null} variable, building upon leave-one-out techniques
\cite{elkaroui13,elkaroui18}, alternatively known as the cavity method
in statistical physics \cite{mezard1987spin}. The crucial issue here
is that it is not clear how to apply leave-one-out techniques when
analyzing non-null coordinates. Consequently, the novelty lies in
recognizing that studying the MLE when covariates are correlated
Gaussian is equivalent to studying the MLE when covariates are
i.i.d.~Gaussian and there is only a single non-null. This connection
is crucial for Theorem \ref{thm:finitemle}, which relies on the new 
\cite[Proposition~B.1]{zhao2021_supp} and Lemma \ref{lemma:exact}.  Theorem
\ref{thm:finitemle} also gives the distribution of linear combinations
for arbitrarily correlated Gaussian covariates---a setting beyond the
reach of the techniques from \cite{sur18}.  

The technique to establish the limiting empirical distribution of the
MLE is a distinct contribution. We show that for every $n$ and $p$,
the MLE can be represented as a function of a correlated Gaussian
vector that has empirical distribution converging to a standard
normal.  This representation \cite[Proposition~B.1]{zhao2021_supp} presented in the supplementary material is new. Prior work \cite{sur18} proved
a version of \eqref{eq:intro-2} utilizing the framework of generalized
approximate message passing
\cite{rangan2011generalized,javanmard2013state,barbier2019optimal}
(G-AMP). These proofs relied on the assumption of i.i.d.~entries on
the design matrix. For correlated matrices, other ideas are needed.

In a broader context, hypothesis testing and confidence interval
construction for high-dimensional regression models have been
extensively studied in the past decade, and even earlier
\cite{buhlmann2011statistics,wainwright2019high}. Here, we mention the
threads of research in the high-dimensional regime
\eqref{eq:asymptotics} that are most relevant for this paper,
deferring a detailed survey to
\cite{montanari2018mean,sur2019thesis}. In
\cite{elkaroui13,elkaroui18,donoho16,thrampoulidis2018precise,bellec2019second,celentano2020lasso},
the authors studied high-dimensional estimation and inference for
linear models using seemingly disparate technical
ingredients---leave-one-out
\cite{elkaroui13,elkaroui18,mezard1987spin}, approximate message
passing (AMP)\cite{donoho2009message,bayati2011lasso}, the Convex
Gaussian Min-Max Theorem (CGMT)
\cite{gordon1988milman,stojnic2013framework,thrampoulidis2015regularized,thrampoulidis2018precise},
second-order Poincare inequalities \cite{chatterjee2009fluctuations}
and debiasing \cite{vandegeer2014,zhang2014,montanari2014}. (See
\cite{jankova2018biased,tony2020semisupervised,ma2020global} for some
other works around debiasing that consider a different
high-dimensional regime.) Utilizing similar techniques,
\cite{sur18,sur19,barbier2019optimal,salehi19} studied
high-dimensional inference for generalized linear models with
i.i.d.~Gaussian designs.

\section{A stochastic representation of the MLE} \label{sec:stochastic}
We consider a setting with $n$ independent observations $(\mbf{x}_i, y_i)_{i=1}^n$ such that the covariates
$\boldsymbol{x}_i\in\R^{p}$ follow a multivariate normal
distribution
$\boldsymbol{x}_i\sim\N(\boldsymbol{0},\boldsymbol{\Sigma})$, 
with covariance $\mbf{\Sigma}\in\R^{p\times p}$, and the response
$y_i\in\{-1, 1\}$ follows the logistic model \eqref{eq:logistic} with
regression coefficients $\boldsymbol{\beta} \in \R^p$. We assume that $\mbf{\Sigma}$ has full column rank so that the model is identifiable. The maximum
likelihood estimator (MLE) $\hat{\bb}$ optimizes the log-likelihood function
\begin{equation}
\label{eq:likelihood}
\ell(\boldsymbol{b}; \boldsymbol{X},\mbf{y})=\sum_{i=1}^n -\log(1+\exp(-y_i \boldsymbol{x}_i^\top \boldsymbol{b}))
\end{equation}
over all $\boldsymbol{b} \in \R^p$. (Here and below, $ \boldsymbol{X}$
is the $n \times p$ matrix of covariates and $\mbf{y}$ the
$n \times 1$ vector of responses.)

\subsection{From dependent to independent covariates}\label{subsec:deptoindep}
We begin our discussion by arguing that the aforementioned setting of
dependent covariates can be translated to that of independent
covariates. This follows from the invariance  of the Gaussian
distribution with respect to linear
transformations.
\begin{prop}\label{prop:indtodep}
Fix any matrix $\mbf{L}$ obeying $\mbf{\Sigma}=\mbf{L}\mbf{L}^{\top}$,
and consider the vectors
\begin{equation}\label{eq:thetabetaeqn}
\hat{\mbf{\theta}}:=\mbf{L}^{\top}\hat{\mbf{\beta}} \qquad \mathrm{and} \qquad \mbf{\theta}:=\mbf{L}^{\top}\mbf{\beta}.
\end{equation}
Then $\hat{\mbf{\theta}}$ is the MLE in a logistic model with regression coefficient $\mbf{\theta}$ and covariates drawn i.i.d.~from $\N(\mbf{0},\mbf{I}_{p})$.
\end{prop}
\begin{proof}
Because the likelihood \eqref{eq:likelihood} depends
on the $\mbf{x}_i$'s and $\mbf{b}$ only through their inner product, 
\begin{equation}\label{eq:transform_mle}
    \ell(\mbf{b}; \mbf{X},\mbf{y})
     =\ell(\mbf{L}^\top\mbf{b}; \mbf{X}\mbf{L}^{-\top}, \mbf{y})
\end{equation}
for every $\mbf{b} \in \R^p$. If $\hat{\mbf{\beta}}$ is the MLE of the original model, then $\hat{\mbf{\theta}} = \mbf{L}^{\top}\hat{\mbf{\beta}}$ is the MLE of a logistic model whose covariates are i.i.d.~draws from $\N(\mbf{0},\mbf{I}_p)$, and true regression coefficients given by $\mbf{\theta} = \mbf{L}^\top\mbf{\beta}$.  
\end{proof}

Proposition \ref{prop:indtodep} has a major consequence---for an
arbitrary variable $j$, which we can assume to be the last variable by permuting the order of the variables, we may choose
$\mbf{\Sigma}=\mbf{L}\mbf{L}^{\top}$ to be a Cholesky factorization of
the covariance matrix, such that $\mbf{x}_i$ can be expressed as
\begin{equation}\label{eq:Lchoice}
 \underbrace{\begin{bmatrix} \star \\ \star \\ \star \\ x_{i,j} \end{bmatrix}}_{\mbf{x}_i} = \underbrace{\begin{bmatrix} \star & & & \\ \star & \star & & \\ \star & \star & \star & \\ \star & \star & \star & \tau_j \end{bmatrix} }_{\mbf{L}} \underbrace{ \begin{bmatrix} \star \\ \star \\ \star \\ z_{i,j} \end{bmatrix}}_{\mbf{z}_i},
\end{equation}
where $\mbf{z}_i \sim \N(\mbf{0},\mbf{I}_p)$ and $\tau_j^2 =\mathrm{Var}(x_{i,j}|\mbf{x}_{i,-j})$. This can be seen from the triangular form: 
\[\mathrm{Var}(x_{i,j}|\mbf{x}_{i,-j}) = \mathrm{Var}(x_{i,j}|\mbf{z}_{i,-j})=  \mathrm{Var}(\tau_j z_{i,j}|\mbf{z}_{i,-j}) =\tau_j^2. \]
Then the equations in \eqref{eq:thetabetaeqn} tell us that 
\begin{align}\label{eq:singlecoordrel}
  \hat{\theta}_j  = \tau_j \hat{\beta}_j,\qquad \theta_j  =  \tau_j \beta_j, 
\end{align}
and, therefore, for any pair $(\alpha, \sigma)$,
\begin{equation}\label{eq:masterrelation}
\tau_j \frac{\hat{\beta}_j - \alpha \beta_j}{\sigma} = \frac{\hat{\theta}_j - \alpha
\theta_j}{\sigma}.
\end{equation}
In particular, if we can find $\alpha$ and $\sigma$ so that the RHS is
approximately $\N(0,1)$, then \eqref{eq:masterrelation} says that the
LHS is approximately $\N(0,1)$ as well. We will use the equivalence
\eqref{eq:thetabetaeqn} whenever possible.

\subsection{A stochastic representation of the MLE}
We work with $\mbf{\Sigma} = \mbf{I}_p$ in this section. The
rotational invariance of the Gaussian distribution in this case yields
an exact stochastic representation for the MLE $\hat{\mbf{\theta}}$,
which is valid for every choice of $n$ and $p$. This representation
will play a crucial role in supporting our subsequent results.

\begin{lemma}\label{lemma:exact}
Let $\hat{\mbf{\theta}}$ denote the MLE in a logistic model with regression vector $\mbf{\theta}$ and covariates drawn i.i.d.~from $\N(\mbf{0},\mbf{I}_p)$. Define the random variables 
 \begin{equation}\label{eq:alphansigmandef}
    \alpha(n) = \frac{\langle \hat{\mbf{\theta}}, \mbf{\theta} \rangle}{\|\mbf{\theta} \|^2} \qquad \mathrm{and} \qquad \sigma^2(n) = \|P_{\mbf{\theta}^\perp} \hat{\mbf{\theta}}\|^2, 
  \end{equation}
  where $P_{\mbf{\theta}^\perp}$ is the projection onto
  $\mbf{\theta}^\perp$, which is the orthogonal complement of
  $\mbf{\theta}$. Then
\[\frac{\hat {\mbf{\theta}} - \alpha(n) \mbf{\theta}}{\sigma(n)} \]
is uniformly distributed on the unit sphere lying in
$\mbf{\theta}^\perp$.
\end{lemma}

\begin{proof}
Notice that 
\[
\hat{\mbf{\theta}} - \alpha(n)\mbf{\theta} = \hat{\mbf{\theta}} - \big\langle \hat{\mbf{\theta}}, \frac{\mbf{\theta} }{\|\mbf{\theta} \|} \big\rangle \frac{\mbf{\theta} }{\|\mbf{\theta} \|} = P_{\mbf{\theta}^\perp} \hat{\mbf{\theta}},
\]
the projection of $\hat{\mbf{\theta}}$ onto the orthogonal
complement of $\mbf{\theta}$. We therefore need to show that for any
orthogonal matrix $\mbf{U}\in\R^{p\times p}$ obeying $\mbf{U} \mbf{\theta} = \mbf{\theta}$, 
\begin{equation} \label{eq:toshow}
 \frac{ \mbf{U}P_{\mbf{\theta}^\perp} \hat{\mbf{\theta}}}{\|P_{\mbf{\theta}^\perp} \hat{\mbf{\theta}} \|} \stackrel{\mathrm{d}}{=} \frac{ P_{\mbf{\theta}^\perp} \hat{\mbf{\theta}}}{{\|P_{\mbf{\theta}^\perp} \hat{\mbf{\theta}} \|} }.
 \end{equation}
We know that
\begin{align}\label{eq:rotationrel}
\hat{\mbf{\theta}}  = P_{\mbf{\theta}} \hat{\mbf{\theta}} + P_{\mbf{\theta}^{\perp}} \hat{\mbf{\theta}}  \implies \mbf{U} \hat{\mbf{\theta}}  = \mbf{U}P_{\mbf{\theta}} \hat{\mbf{\theta}} + \mbf{U }P_{\mbf{\theta}^{\perp}} \hat{\mbf{\theta}}  =  P_{\mbf{\theta}}\hat{\mbf{\theta}} + \mbf{U }P_{\mbf{\theta}^{\perp}} \hat{\mbf{\theta}},
\end{align}
where the last equality follows from the definition of $\mbf{U}$. Now, $\mbf{U} \hat{\mbf{\theta}}$ is the MLE in a logistic model with covariates drawn i.i.d.~from $\N(\mbf{0}, \mbf{I}_{p})$ and regression vector $\mbf{U} \mbf{\theta} = \mbf{\theta}$. Hence, $\mbf{U} \hat{\mbf{\theta}} \stackrel{\mathrm{d}}{=} \hat{\mbf{\theta}}$ and \eqref{eq:rotationrel} leads to 
\[\mbf{U }\frac{P_{\mbf{\theta}^{\perp}} \hat{\mbf{\theta}}}{\|P_{\mbf{\theta}^\perp} \hat{\mbf{\theta}} \|} \stackrel{\mathrm{d}}{=} \frac{\hat{\mbf{\theta}} - P_{\mbf{\theta}}\hat{\mbf{\theta}}}{\|P_{\mbf{\theta}^\perp} \hat{\mbf{\theta}} \|}  = \frac{ P_{\mbf{\theta}^{\perp}}\hat{\mbf{\theta}} }{\|P_{\mbf{\theta}^\perp} \hat{\mbf{\theta}} \|}.\]
\end{proof}

\section{The asymptotic distribution of the MLE in high dimensions}\label{sec:main}

We now study the distribution of the MLE in the limit of
a large number of variables and observations. We consider a sequence of logistic regression problems with $n$ observations and $p(n)$ variables. In each problem instance, we have $n$ independent observations $(\mbf{x}_i, y_i)_{i=1}^n$ from a logistic model with covariates $\boldsymbol{x}_i\sim\N(\boldsymbol{0},\boldsymbol{\Sigma}(n)), \mbf{\Sigma}(n)\in\R^{p(n)\times p(n)}$, regression coefficients $\boldsymbol{\beta}(n)$, and response $y_i\in\{-1, 1\}$. As the sample size increases, we assume that the dimensionality $p(n)/n$ approaches a fixed limit in the sense that
\begin{equation}\label{eq:asymptotics}
p(n)/n \,\, \to \,\, \kappa \,\, > 0. 
\end{equation}

As in
\cite{sur18}, we consider a scaling of the regression coefficients
obeying
\begin{equation}\label{eq:gammaequation}
\var(\mbf{x}^\top_i\mbf{\beta}(n)) =\mbf{\beta}(n)^\top \mbf{\Sigma}(n)\mbf{\beta}(n) \,\, \to \,\, \gamma^2 < \infty. 
\end{equation}
This scaling keeps the ``signal-to-noise-ratio'' fixed. The larger
$\gamma$, the easier it becomes to classify the observations.
(If the parameter $\gamma$ were allowed to diverge to infinity, we
would have a noiseless problem in which we could correctly classify
essentially all the observations.)

In the remainder of this paper, we will drop $n$ from expressions such
as $p(n)$, $\mbf{\beta}(n)$, and $\mbf{\Sigma}(n)$ to simplify the
notation. We shall however remember that the number of variables $p$
grows in proportion to the sample size $n$.

\subsection{Existence of the MLE}\label{sec:exist-mle}

\begin{figure}
\begin{center}
          \includegraphics[width=6.5cm,keepaspectratio]{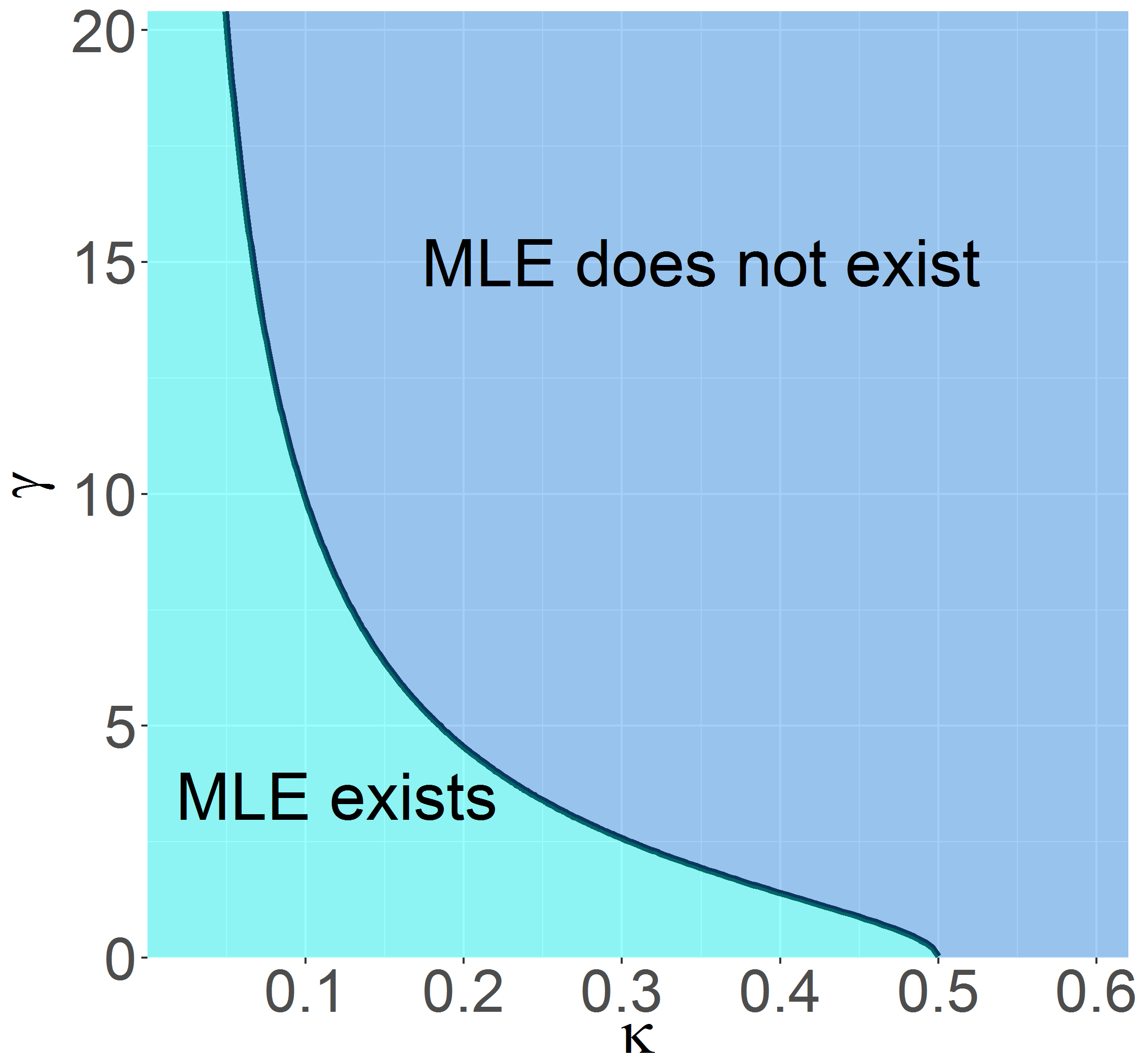}
\end{center}
\caption{Boundary
  curve $\kappa \mapsto \gmle(\kappa)$ separating the regions where
  the MLE asymptotically exists and where it does not \cite[Figure 1(a)]{candes18}.}
\label{fig:formula}
\end{figure}

An important issue in logistic
regression is that the MLE does not always exist. In fact, the MLE
exists if and only if the cases (the points $\bx_i$ for which
$y_i = 1$) and controls (those for which $y_i = -1$) cannot be
linearly separated; linear separation here means that there is a
hyperplane such that all the cases are on one side of the plane and
all the controls on the other.

When both $n$ and $p$ are large, whether such a separating hyperplane
exists depends only on the dimensionality $\kappa$ and the overall
signal strength $\gamma$. In the asymptotic setting described above,
\cite[Theorem 1]{candes18} demonstrated a phase transition phenomenon:
there is a curve $\gamma = \gmle(\kappa)$ in the $(\kappa, \gamma)$
plane that separates the region where the MLE exists from that where
it does not, see Figure \ref{fig:formula}. Formally,
\begin{align*}
    \gamma>g_{\mathrm{MLE}}(\kappa) \implies \lim_{n,p\to\infty} \prob(\mathrm{MLE \: exists})\to 0,\\
    \gamma<g_{\mathrm{MLE}}(\kappa) \implies \lim_{n,p\to\infty} \prob(\mathrm{MLE\:  exists})\to 1.
\end{align*}
It is noteworthy that the phase-transition diagram only depends on
whether $\gamma > \gmle(\kappa)$ and, therefore, does not depend on
the details of the covariance $\mbf{\Sigma}$ of the covariates. Since
we are interested in the distribution of the MLE, we shall consider
values of the dimensionality parameter $\kappa$ and signal strength $\gamma$ in
the light blue region from Figure \ref{fig:formula}.

\subsection{Finite-dimensional marginals of the MLE}\label{sec:finitemle}
We begin by establishing the asymptotic behavior of the random
variables $\alpha(n)$ and $\sigma(n)$, introduced in
\eqref{eq:alphansigmandef}. These limits will play a key role in the
distribution of the MLE.

\begin{lemma}\label{lemma:as_conv2}
Consider a sequence of logistic models with covariates drawn i.i.d.~from $\N(\mbf{0},\mbf{I}_p)$ and regression vectors $\mbf{\theta}$ satisfying $\lim_{n\to\infty} \|\mbf{\theta}\|^2\, \to \,\gamma^2$. Let $\hat{\mbf{\theta}}$ be the MLE and define $\alpha(n)$ and $ \sigma(n)$ as in \eqref{eq:alphansigmandef}. Then, if $(\kappa,\gamma)$ lies in the region where the MLE exists asymptotically, we have that
\begin{equation}\label{eq:aslimits}
    \alpha(n) \,\, \stackrel{\mathrm{a.s.}}{\longrightarrow} \,\,
    \alpha_\star \qquad  \mathrm{and} \qquad \sigma(n)^2 \,\, 
    \stackrel{\mathrm{a.s.}}{\longrightarrow} \,\, \kappa \sigma_\star^2,
\end{equation}
where $\alpha_{\star}$ and $\sigma_\star$ are numerical constants that
only depend on $\kappa$ and $\gamma$. 
\end{lemma}
We defer the proof to the supplementary material \cite[Section~A]{zhao2021_supp}; here, we explain
where the parameters $(\alpha_\star, \sigma_\star)$ come from. Along
with an additional parameter $\lambda_\star$, the triple
$(\alpha_\star, \sigma_\star,\lambda_\star)$ is the unique solution to
the system of equations parameterized by $(\kappa, \gamma)$ in three
variables $(\alpha, \sigma, \lambda)$ given by
\begin{equation}\label{eq:system_equation}
    \begin{dcases}
\sigma^2 & =\frac{1}{\kappa^2}\E{2\rho'(Q_1)(\lambda\rho'(\prox_{\lambda\rho}(Q_2)))^2}\\
0 & =\E{\rho'(Q_1)Q_1\lambda\rho'(\prox_{\lambda\rho}(Q_2))}\\
1-\kappa & = \E{\frac{2\rho'(Q_1)}{1+\lambda\rho''(\prox_{\lambda\rho}(Q_2))}},
\end{dcases}
\end{equation}
where $(Q_1,Q_2)$ is a bivariate normal variable with mean $\mbf{0}$ and covariance
\[
\mbf{\Sigma}(\alpha,\sigma)=\left[\begin{matrix}
\gamma^2 & -\alpha\gamma^2\\
-\alpha\gamma^2 & \alpha^2\gamma^2+\kappa\sigma^2
\end{matrix}\right]. 
\]
Above, the proximal operator is defined as
\[
\prox_{\lambda\rho}(z)=\arg\min_{t\in\R}\left\{\lambda\rho(t)+\frac{1}{2}(t-z)^2\right\},
\]
where $\rho(t)=\log(1+e^t)$.
This system of equations can be rigorously derived from the
generalized approximate message passing algorithm
\cite{rangan2011generalized,javanmard2013state}, or by analyzing the
auxiliary optimization problem \cite{salehi19}. They can also be
heuristically understood with an argument similar to that in
\cite{elkaroui13}; we defer to \cite{sur18} for a complete
discussion. The important point here is that in the region where the
MLE exists, the system \eqref{eq:system_equation} has a unique
solution.  Lemma \ref{lemma:as_conv2} contributes novel insights by
interpreting $\alpha_\star, \sigma_\star$ through the lens of the
finite sample quantities $\alpha(n), \sigma(n)$, which proves crucial
for establishing Theorem 3.1.

We are now in a position to describe the asymptotic behavior of the
MLE. The proof is deferred to  the supplementary material \cite[Section~A]{zhao2021_supp}.
\begin{theorem}\label{thm:finitemle}
  Consider a logistic model with covariates $\mbf{x}_i$ drawn
  i.i.d.~from $\mathcal{N}(\mbf{0},\mbf{\Sigma}(n))$ and assume we are
  in the $(\kappa,\gamma)$ region where the MLE exists asymptotically.
  Then for every coordinate whose regression coefficient satisfies
   $\sqrt{n} \tau_j\beta_j = O(1)$, 
    \begin{equation}\label{eq:nonnullsingle}
   \frac{\sqrt{n}(\hat \beta_j - \alpha_\star
  \beta_j)}{\sigma_\star/\tau_j } \,  \stackrel{\mathrm{d}}{\longrightarrow} \, \N(0,1). 
\end{equation}
Above $\tau^2_j = \var(x_{i,j}\,|\,\mbf{x}_{i,-j})$ is the conditional
variance of $x_{i,j}$ given all the other covariates. More generally,
for any sequence of deterministic unit normed vectors $\mbf{v}(n)$
with $\sqrt{n} \tau(\mbf{v}) \mbf{v}(n)^{\top}\mbf{\beta}(n) = O(1)$,
we have that
  \begin{equation}\label{eq:lincomb}
    \frac{\sqrt{n}\mbf{v}^\top(\hat{\mbf{\beta}} - \alpha_\star
      \mbf{\beta})}{\sigma_\star/\tau(\mbf{v})} \, \stackrel{\mathrm{d}}{\longrightarrow} \, \N(0,1).
\end{equation}
Here $\tau(\mbf{v})$ is given by
\[\tau^2(\mbf{v}) = \mathrm{Var}(\mbf{v}^{\top} \mbf{x}_i | P_{\mbf{v}^{\perp}} \mbf{x}_i )= \left(\mbf{v}^{\top} \mbf{\Theta}(n) \mbf{v} \right)^{-1 }, \]
where $\mbf{\Theta}(n)$ equals the precision matrix
$\mbf{\Sigma}(n)^{-1}$.  A consequence is this: consider a finite set
of coordinates $\mathcal{S} \subset \{1, \ldots, p\}$ obeying
$\sqrt{n}({\boldsymbol{\beta}_{\mathcal{S}}^\top \boldsymbol{\Theta}_{\mathcal{S}}^{-1} \mbf{\beta}_{\mathcal{S}}})^{\frac{1}{2}} = O(1)$.
Then
\begin{equation}\label{eq:nonnull}
  \frac{\sqrt{n}\mbf{\Theta}_{\mathcal{S}}^{-1/2}(\hat{\mbf{\beta}}_{\mathcal{S}} -
  \alpha_{\star} \mbf{\beta}_{\mathcal{S}})}{\sigma_\star}
\, \, \stackrel{\mathrm{d}}{\longrightarrow}\,\,  \N(\mbf{0},\mbf{I}_{|\mathcal{S}|}).
\end{equation}
Above $ \mbf{\beta}_{\mathcal{S}}$ is the slice of $\mbf{\beta}$ with
entries in $\mathcal{S}$ and, similarly,
$\mbf{\Theta}_{\mathcal{S}}$ is the slice of the precision
matrix $\mbf{\Theta}$ with rows and columns in $\mathcal{S}$. 
\end{theorem}

Returning to the Introduction, we now see that the behavior of the MLE
is different from that implied by the classical textbook result, which
states that
\[
 \sqrt{n} (\hat{\mbf{\beta}} - \mbf{\beta}) \, \,
 \stackrel{\mathrm{d}}{\longrightarrow}\,\, \N(0,
 \mathcal{I}^{-1}_\beta). 
 \]
 We also see that Theorem \ref{thm:finitemle} extends
 \cite[Theorem~3]{sur18} in multiple directions. Indeed, this prior
 work assumed standardized and independent covariates
 (i.e.~$\mbf{\Sigma} = \mbf{I}$)---implying that $\tau_j^2 = 1$---and
 established $\sqrt{n}\hat{\beta}_j/\sigma_\star \convd \N(0,1)$ only
 in the special case $\beta_j=0$. 
 
\subsubsection{Finite sample accuracy} 
We study the finite sample accuracy of Theorem \ref{thm:finitemle}
through numerical examples.
We consider an experiment with a fixed number of observations set to
$n=4,000$
and a number of variables set to $p=800$
so that $\kappa
= 0.2$. We set the signal strength to
$\gamma^2=5$.
(For this problem size, the asymptotic result for null variables has
been observed to be very accurate when the covariates are independent
\cite{sur18}.) 

We sample the covariates such that the covariance matrix is the correlation matrix from an AR(1) model with parameter $\rho=0.5$, i.e. $\Sigma_{ij} = \rho^{|i-j|}$.
We then randomly sample half of the coefficients to be non-nulls, with
equal and positive magnitudes, chosen to attain the desired signal
strength $\mbf{\beta}^{\top}\mbf{\Sigma}\mbf{\beta}
= 5$. For a given non-null coordinate
$\beta_j$,
we calculate the centered and scaled MLE $\zt_j$
\eqref{eq:Z}, and repeat the experiment $B=100,000$
times. Figure \ref{fig:qqplot} shows a qqplot of the empirical
distribution of $ \zt_j$
versus the standard normal distribution. Observe that the quantiles
align perfectly, demonstrating the accuracy of
\eqref{eq:nonnullsingle}.

We further examine the empirical accuracy of \eqref{eq:nonnullsingle}
through the lens of confidence intervals and finite sample
coverage. Theorem \ref{thm:finitemle} suggests that if
$z_{(1-\alpha/2)}$ is the $(1-\alpha/2)$th quantile of a standard
normal variable, $\beta_j$ should lie within the interval
\begin{equation}\label{eq:cisingle}
    \left[\frac{1}{\alpha_\star}\left(\hat{\beta}_j-\frac{\sigma_\star}{\sqrt{n}\tau_j}z_{(1-\alpha/2)}\right),\frac{1}{\alpha_\star}\left(\hat{\beta}_j+\frac{\sigma_\star}{\sqrt{n}\tau_j}z_{(1-\alpha/2)}\right)\right]
\end{equation}
about $(1-\alpha)B$ times. Table \ref{tab:cov_prop_nonnull} shows the proportion of experiments in which $\beta_j$ is covered by \eqref{eq:cisingle} for different choices of the confidence level $(1-\alpha)$, along with the respective standard errors. For every confidence level,  the empirical coverage proportion lies extremely close to the corresponding target.

\begin{figure}
\begin{center}
          \includegraphics[width=6.5cm,keepaspectratio]{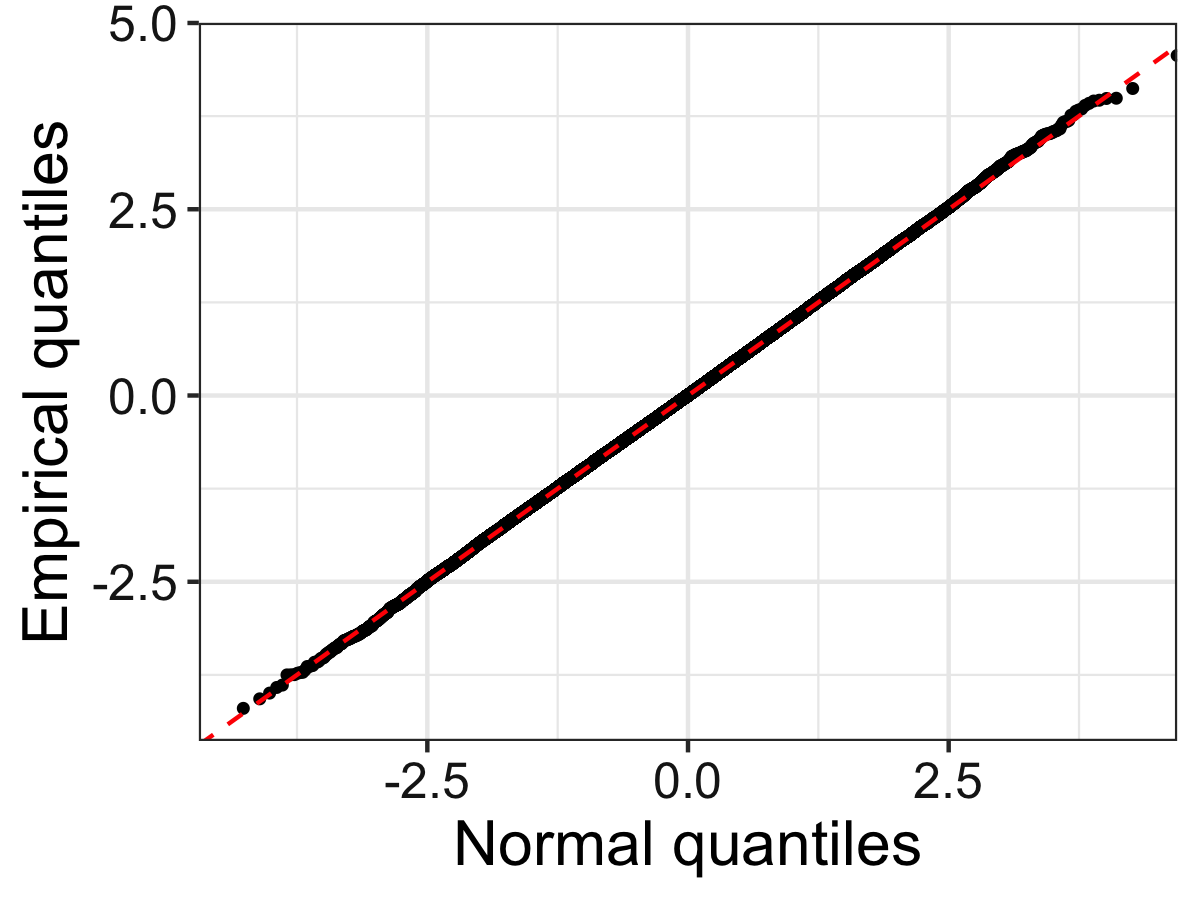}
\end{center}
\caption{Quantiles of the empirical distribution of an appropriately centered and scaled MLE coordinate versus standard normal quantiles. }
\label{fig:qqplot}
\end{figure}

\begin{table*}
\caption{Coverage proportion of a single variable. Each cell reports the proportion of times $\beta_j$ falls within \eqref{eq:cisingle}, calculated over $B=100,000$ repetitions; the standard errors are provided as well. }
\label{tab:cov_prop_nonnull}
\begin{tabular}{@{}cccccc@{}}
 \hline
     Nominal coverage & & & & & \\
$100(1-\alpha)$ & 99 & 98 & 95 & 90 & 80 \\
        \hline
        Empirical coverage & 98.97 & 97.96 & 94.99 & 89.88 & 79.88 \\
        Standard error & 0.03 & 0.04 & 0.07 & 0.10 & 0.13 \\
      \hline
\end{tabular}
\end{table*}

\subsubsection{Condition on the regression coefficients}

In this section, we conduct simulations to investigate to what extent
the condition on the magnitude of $\beta_j$ in Theorem
\ref{thm:finitemle} may be relaxed. We vary the magnitude of the
non-null coefficients $\beta_j$ by choosing $\beta_j$ such that
$\tau_j \beta_j/\gamma \in\{ 0.15, 0.3, 0.5\}$ (so that
$\sqrt{n}\tau_j\beta_j$ is large). We report the variance of a single
MLE coordinate in $B = 10,000$ repeated experiments (Figure
\ref{fig:newfig}). The observed biases range
  from 1.450 to 1.461 in the simulations and do not show a pattern. In this example, we set
$\kappa = 0.2$, $\gamma = 2$ and use the same covariance matrix as in
the last paragraph. The theoretical standard deviation of
$\hat{\beta}_j$ is thus 5.65 (dashed line). We observe that our theory
works well when $\tau_j\beta_j/\gamma \leq 0.15$, because the
empirical standard error is close to the theoretical prediction and
the MLE is approximately Gaussian (Figure
\ref{fig:qqplot_large_nonnull}). However, the standard error of the
MLE increases as $\beta_j$ increases. At $\tau_j\beta_j/\gamma = 0.5$,
for instance, the observed standard error when $n=4000$ is 7.71, which
is 36\% percent larger than the theoretical standard deviation. We
also observe that for a fixed $\gamma$ and $\kappa$, the standard
error slightly decreases as $n$ increases, suggesting that the theory
becomes more accurate at larger $n$ if
$\tau_j\beta_j/\gamma \lesssim 0.15$.

Now that we have reasons to believe the theory holds for
$\tau_j \beta_j$ that is not vanishing, we study a possible path to a sharper statement. Theorem \ref{thm:finitemle} requires  the condition
$\sqrt{n}\tau_j\beta_j = O(1)$ to ensure that
$\sqrt{n}(\alpha(n) - \alpha_\star) \tau_j \beta_j = o(1)$. (See the
proof in  \cite[Section~A]{zhao2021_supp}.) We thus study the behavior of the random variable
$\sqrt{n}(\alpha(n) - \alpha_\star)$ and compute its standard
deviation. We use the same covariance matrix as before and set half of
the variables to be non-nulls. We compute $\alpha(n)$ as
$ \alpha(n) = \hat{\beta}^\top \Sigma \beta/\gamma^2, $ and report
\begin{equation} \label{eq:condition_sim}
    \widehat{\mathrm{sd}}(\sqrt{n}(\alpha(n)-\alpha_\star)) = \left(\frac{1}{m}\sum_{i=1}^m \left[\sqrt{n}(\alpha_i(n) -\alpha_\star)\right]^2 \right)^{1/2}
\end{equation}
over $m = 10,000$ repetitions (Figure \ref{fig:condition}). 
We observe that while the standard deviations
differ for varied choices of $\kappa$ and $\gamma$, they slightly decrease as $n$ increases and approach a constant asymptote. This suggests that $\sqrt{n}(\alpha(n)-\alpha_\star) = O_P(1)$ and the condition may be relaxed to $\tau_j\beta_j = o(1)$.   We do not expect that this latter condition can be relaxed further. To see this, note that this
condition $\tau_j\beta_j = o(1)$ was required for
\cite[Theorem~2]{sur18}, since this earlier result holds in the
setting where the empirical distribution of $\beta$ converges weakly
to a distribution with finite second moment. In the setting of
\cite{sur18}, $\tau_j = 1/\sqrt{n}$ and therefore, their condition
implies that $\tau_j\beta_j = o(1)$ for every coordinate. 

\begin{figure*}
    \begin{subfigure}[t]{0.5\textwidth}
        \centering
        \includegraphics[height=4cm,keepaspectratio]{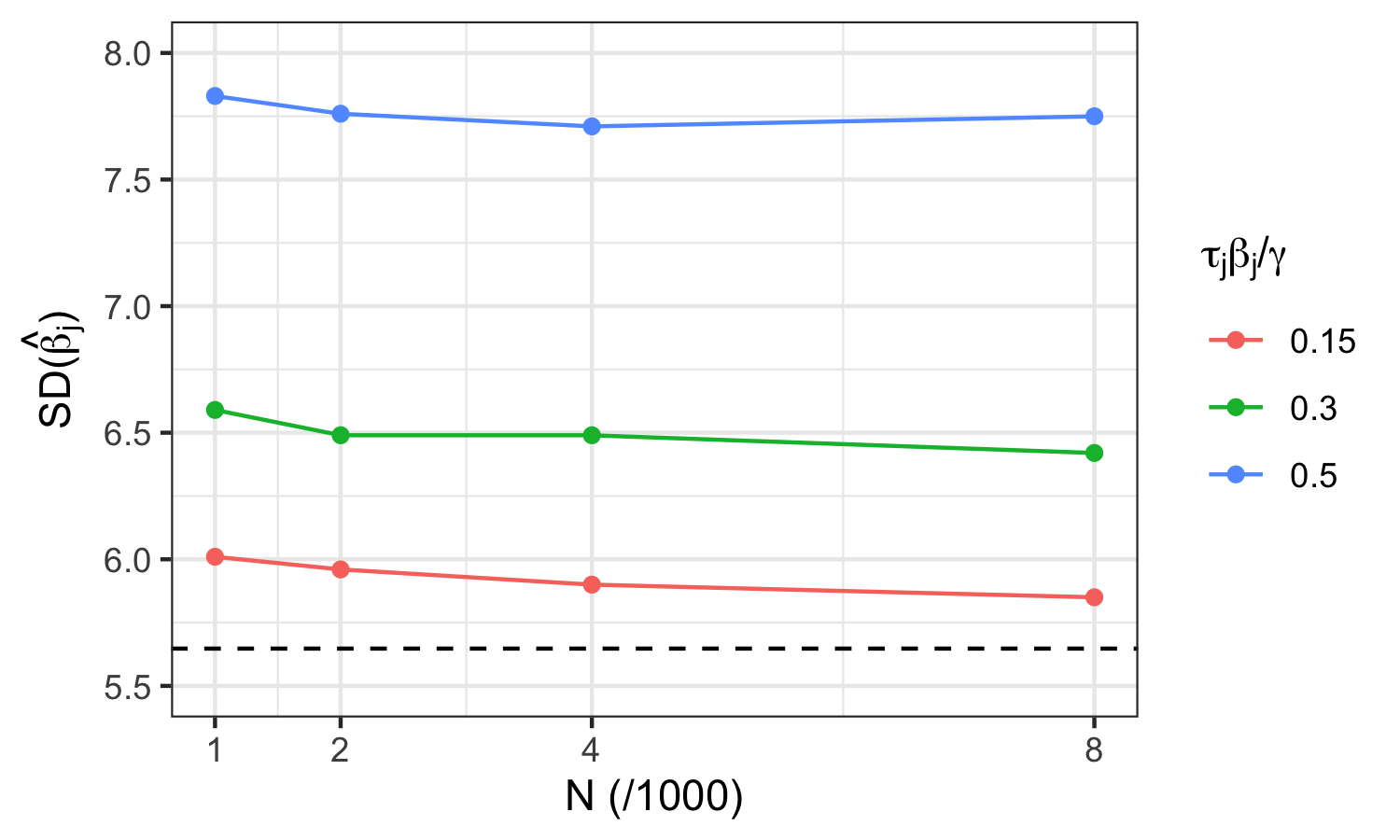}
        \caption{}
        \label{fig:sd_largebeta}
    \end{subfigure}%
        \hspace*{0em}
    \begin{subfigure}[t]{0.5\textwidth}
        \centering
        \includegraphics[height=4cm,keepaspectratio]{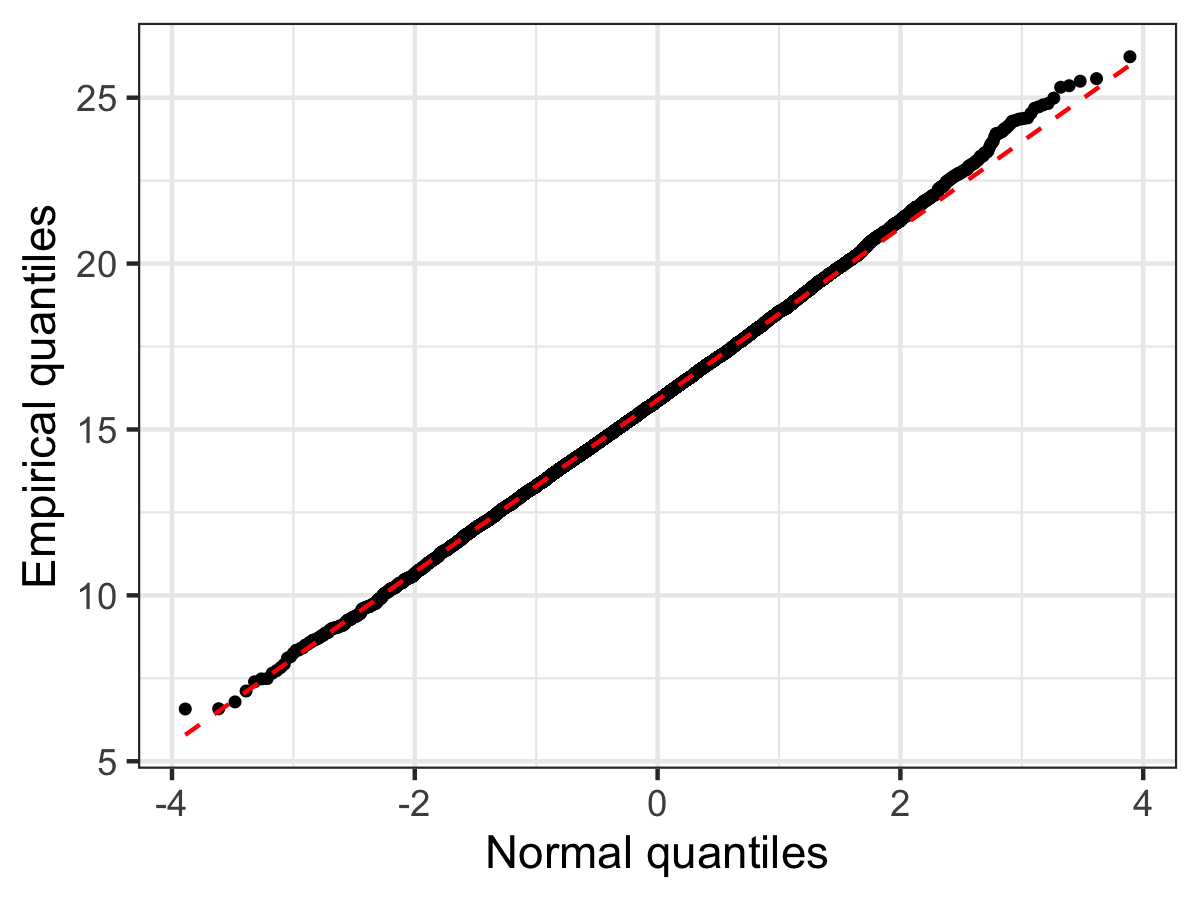}
        \caption{}
        \label{fig:qqplot_large_nonnull}
    \end{subfigure}%
    \caption{(a) Standard deviation of a single MLE $\hat{\beta}_j$ versus the sample size $n$ (shown on the $x$-axis as $n/1000$) when size of the coefficient $\beta_j$ varies at $\tau_j \beta_j/\gamma \in\{ 0.15, 0.3, 0.5\}$. $\mathrm{SD}(\hat{\beta}_j)$ is larger than than theoretical value (dashed line) when $\beta_j$ is large. The standard deviations are calculated from $B = 10,000$ repetitions. (b) Normal quantile plot for the non-null variable with $\tau_j \beta_j / \gamma = 0.15$ and $n=4000$ in Figure (a).}
    \hspace*{0em}
    \label{fig:newfig}
\end{figure*}

\begin{figure*}
    \centering
    \includegraphics[height = 4cm,keepaspectratio]{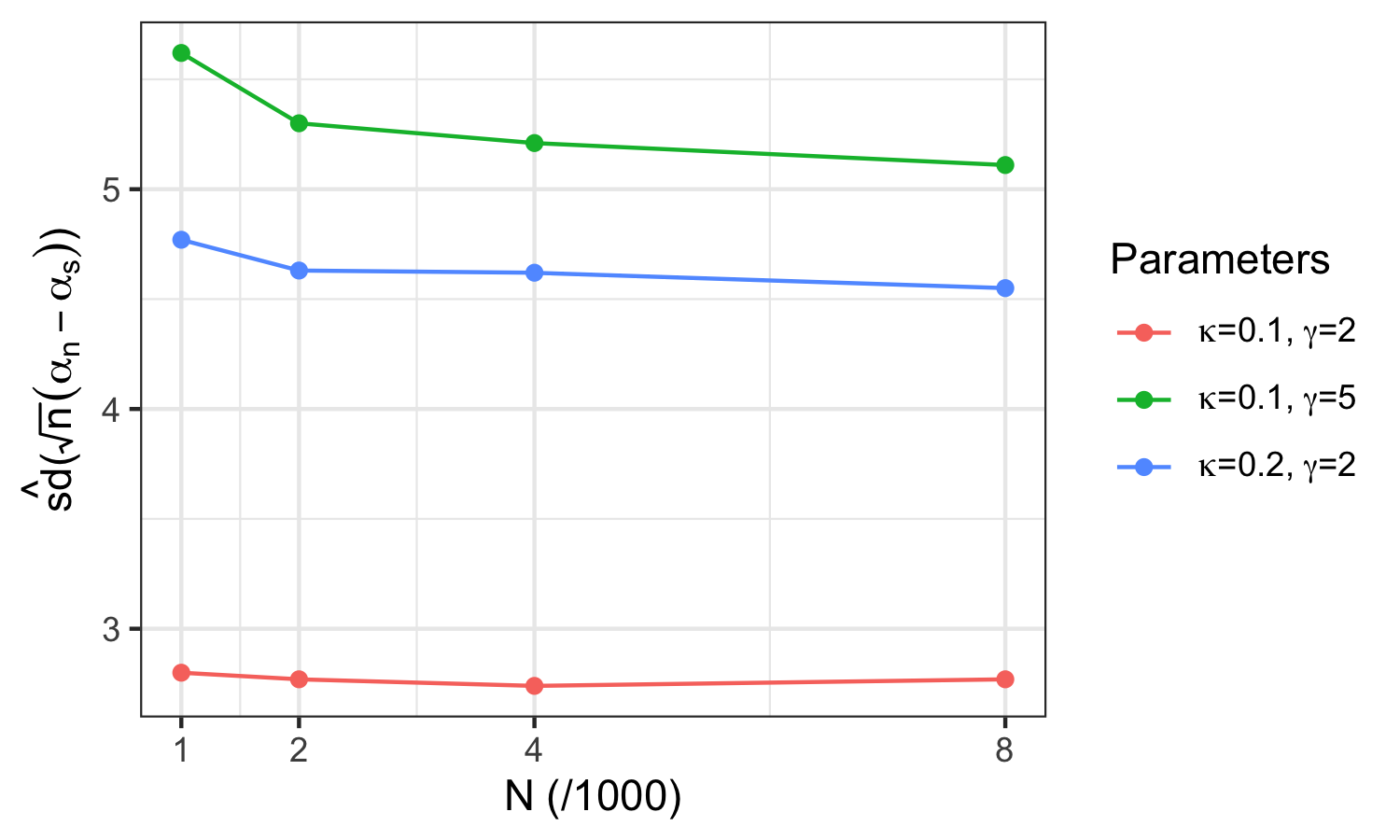}
    \caption{Plot of $\widehat{\mathrm{sd}}(\sqrt{n}(\alpha(n)-\alpha_\star))$ (see Eqn. \eqref{eq:condition_sim}) versus the sample size $n$ different parameters $\kappa$ and $\gamma$.  }
    \label{fig:condition}
\end{figure*}

\subsection{The empirical distribution of all MLE
    coordinates (single experiment)}\label{sec:bulk}

  The preceding section tells us about the finite-dimensional
  marginals of the MLE, e.g.~about the distribution of a given
  coordinate $\hat{\beta}_j$ when we repeat experiments. We now turn
  to a different type of asymptotics and characterize the limiting
  empirical distribution of {\em all} the coordinates calculated from
  a single experiment.
    
\begin{theorem} \label{thm:bulk} Let
    $c(n) = \lambda_{\max}
    (\mbf{\Sigma}(n))/\lambda_{\min}(\mbf{\Sigma}(n))$
    be the condition number of $\mbf{\Sigma}(n)$. Assume that
    $\limsup_{n\to\infty} \, c(n) < \infty$,
  and
    that $(\kappa,\gamma)$ lies in the region where the MLE exists
    asymptotically. Then the empirical cumulative distribution
    function of the rescaled MLE \eqref{eq:Z} converges pointwise to
    that of a standard normal distribution, namely, for each
    $t \in \R$,
    
\begin{equation}\label{eq:bulk_dep}
 \frac{1}{p}\sum_{i=1}^p \ind{ \zt_j \leq  t}\stackrel{\mathrm{P}}{\longrightarrow}\Phi(t).
\end{equation}
\end{theorem}

As explained in the
Introduction, this says that empirical averages of functionals of the
marginals have a limit \eqref{eq:weakconv}.  By \cite[Corollary~B.2]{zhao2021_supp},
this implies that the empirical distribution of
$\mbf{\zt}$ converges weakly to a standard Gaussian, in
probability.

The statement above can be extended to general testing functions 
(the proof is provided in the supplementary material \cite[Section~C]{zhao2021_supp}):
\begin{theorem}\label{thm:general_testfun}
Consider any pseudo-Lipschitz function of order 2 \footnote{A
function $\psi\,:\,\R^m\to \R$ is said to be pseudo-Lipschitz of
  order $k$ if there exists a constant $L>0$ such that for all
  $\mbf{t}_0, \mbf{t}_1\in\R^m$,
  $\|\psi(\mbf{t}_0) - \psi(\mbf{t}_1)\| \leq L\,
  (1+\|\mbf{t}_0\|^{k-1}+\|\mbf{t}_1\|^{k-1}) \, \|\mbf{t}_0 -
  \mbf{t}_1\|$.}
 and suppose the conditions of Theorem \ref{thm:finitemle} hold. Further, assume that the empirical distribution sequence $\sum_{j=1}^p\delta_{\sqrt{n} \tau_j \beta_j}/p$  converges weakly to a distribution $\Pi$ with finite second moment and that $\frac{1}{p}\sum_{j=1}^p n\tau_j^2\beta_j^2 \stackrel{p}{\longrightarrow} \E{\eta^2}$ for $\eta\sim\Pi$, then 
\begin{equation}
\frac{1}{p}\sum_{i=1}^p \psi(\sqrt{n} \tau_j (\hat{\beta}_j - \alpha_\star \beta_j), \sqrt{n} \tau_j \beta_j) \stackrel{P}{\longrightarrow} \E{\psi(\sigma_\star Z, \eta)}. 
\end{equation}
\end{theorem}

In the remainder of this section, we study the empirical
accuracy of Theorem \ref{thm:bulk} in finite samples through some
simulated examples.

The nature of this result is similar to \cite[Theorem~2]{sur18}. However, the techniques from \cite{sur18} cannot be applied
here. To prove Theorem \ref{thm:bulk}, we connect the rescaled MLE
vector $\mbf{T}$ with a correlated Gaussian vector through a
stochastic representation result, similar in spirit to Lemma
\ref{lemma:exact}. This is done in \cite[Proposition~B.1]{zhao2021_supp} and \cite[Corollary~B.1]{zhao2021_supp}. The rest of the proof focuses on analyzing this representation, see \cite[Section~B]{zhao2021_supp}.

\subsubsection{Finite sample accuracy} We consider an experiment with
dimensions $n$ and $p$ the same as that for Figure \ref{fig:qqplot},
and the regression vector sampled similarly.
According to \eqref{eq:bulk_dep}, about $(1-\alpha)$ of all the $\beta_j$'s should lie within the
corresponding intervals \eqref{eq:cisingle}. 

Table \ref{tab:cov_prop} shows the proportion of $\beta_j$'s covered
by these intervals for a few commonly used confidence
levels $(1-\alpha)$.\footnote{Note that, in Table \ref{tab:cov_prop_nonnull} we investigated a single coordinate across many replicates, whereas here we consider all the coordinates in each instance.} 
The proportions are as predicted for each of the confidence levels, and
every covariance we simulated from.

The four columns in Table \ref{tab:cov_prop} correspond to different
covariance matrices; they include a random correlation matrix (details
are given below), a correlation matrix from an AR(1) model whose
parameter is either set to $\rho=0.8$ or $\rho=0.5$, and a covariance
matrix set to be the identity.  The random correlation matrix is
sampled as follows: we randomly pick an orthogonal matrix $\mbf{U}$,
and eigenvalues $\lambda_1,\dots, \lambda_p$ i.i.d.~from a chi-squared
distribution with 10 degrees of freedom. We then form a positive
definite matrix $\mbf{B} = \mbf{U}^\top\mbf{\Lambda}\mbf{U}$ from
these eigenvalues, where
$\mbf{\Lambda} = \mathrm{diag}(\lambda_1,\dots, \lambda_p)$.
$\mbf{\Sigma}$ is the correlation matrix obtained from $\mbf{B}$ by
scaling the variables to have unit variance.

\begin{table*}
\caption{Each cell reports the proportion of \emph{all} the
      variables in each run falling within the corresponding intervals
      from \eqref{eq:cisingle}, averaged over $B=100,000$ repetitions;
      the standard deviation is given between parentheses.}
\label{tab:cov_prop}
\begin{tabular}{@{}ccccc@{}}
 \hline
        Nominal coverage & & & &\\
$100(1-\alpha)$ & Random & $ \rho=0.8$ & $\rho=0.5$ & Identity \\
        \hline
      99  & 99.178 (0.002) & 99.195 (0.002) & 99.187 (0.002)& 99.175 (0.002)\\
      98  &97.873 (0.003) & 97.908 (0.003) & 97.890 (0.003)& 97.865 (0.003)\\
      95  & 94.826 (0.005)& 94.884 (0.005)&
      94.857 (0.005)& 94.811 (0.005)\\
      90  & 89.798 (0.007)& 89.883 (0.007)& 
      89.847 (0.007)& 89.780 (0.007)\\
      80  &  79.784 (0.009)& 79.896 (0.009) &
      79.837 (0.009)& 79.751 (0.009)\\
      \hline
\end{tabular}
\end{table*}

\section{The distribution of the LRT}\label{sec:lrt}
 Lastly, we study the distribution of the
log-likelihood ratio (LLR) test statistic
\begin{equation}\label{eq:llr-defn}
    \llr_j = \max_{\mbf{b}} \ell(\mbf{b};\mbf{X},\mbf{y}) -
    \max_{\mbf{b}:\: {b}_j=0} \ell(\mbf{b};\mbf{X},\mbf{y}), 
  \end{equation}
  which is routinely used to test whether the $j$th variable is in the
  model or not; i.~e.~whether $\beta_j=0$ or not. 
  
 \begin{theorem}\label{thm:llr}
  Assume that we are in the $(\kappa,\gamma)$ region where the MLE
  exists asymptotically. Then under the null (i.e.~$\beta_j = 0$), 
\[
    2\llr_j \,\, \convd \,\, \frac{\kappa\sigma_\star^2}{\lambda_\star}\chi^2_1.
  \]
Further, for every finite $\ell$, twice the LLR for testing
  $\ell$ null hypotheses $\beta_{j_1}=\dots=\beta_{j_\ell}=0$ is
  asymptotically distributed as
  $(\kappa\sigma_\star^2/\lambda_\star)\,\, \chi^2_\ell$.
\end{theorem}

Invoking results from Section \ref{subsec:deptoindep}, we will show that this
is a rather straightforward extension of  \cite[Theorem~4]{sur18}, which  deals with independent covariates, see also \cite[Theorem~1]{sur19}. Choosing $\mbf{L}$ to be the same as in \eqref{eq:Lchoice} after permuting $\beta_j$ to be the last variable and setting
$\mbf{b}'=\mbf{L}^{\top}\mbf{b}$ tell us that $b'_j=0$ if and only if $b_j=0$. Hence, $\mathrm{LLR_j}$ reduces to 
\begin{align}\label{eq:llrred}
    \llr_j & =  \max_{\mbf{b}} \ell(\mbf{L}^\top\mbf{b}; \mbf{X}\mbf{L}^{-\top}, \mbf{y}) - \max_{\mbf{b}: b_j = 0} \ell(\mbf{L}^\top\mbf{b}; \mbf{X}\mbf{L}^{-\top}, \mbf{y}) \\
 &   = \max_{\mbf{b}'} \ell(\mbf{b}'; \mbf{X}\mbf{L}^{-\top},\mbf{y}) -\max_{\mbf{b}': b'_j=0} \ell(\mbf{b}'; \mbf{X}\mbf{L}^{-\top},\mbf{y}),
\end{align}
which is the  log-likelihood ratio statistic in a model
with covariates drawn i.i.d.~from $\N(\mbf{0},\mbf{I}_p)$ and 
regression coefficient given by
$\mbf{\theta}=\mbf{L}^\top \mbf{\beta}$. This in turn satisfies $\theta_j = 0$ if and only if $\beta_j=0$ so that we
can think of the LLR above as testing $\theta_j = 0$.
Consequently, the asymptotic
distribution is the same as that given in
\cite[Theorem~4]{sur18} with
$\gamma^2 = \lim_{n \rightarrow \infty} \|\mbf{\theta}\|^2 = \lim_{n
  \rightarrow \infty} \mbf{\beta}^\top \mbf{\Sigma} \mbf{\beta}$.
The equality of the likelihood ratios implies that to study the
finite sample accuracy of Theorem \ref{thm:llr}, we may just as
well assume we have independent covariates; hence, we refer the
readers to \cite{sur18} for empirical results detailing the quality of
the rescaled chi-square approximation in finite samples.

\section{Accuracy  with estimated parameters}\label{sec:empirical}

In practice, the signal strength $\gamma^2$ and conditional variance
$\tau_j^2$ are typically not known a priori. In this section, we plug
in estimates of these quantities and investigate their empirical
performance. We focus on testing a null variable and constructing
confidence intervals.

The parameters are the same as in Section \ref{sec:finitemle}. In brief, we
set $n=4,000$, $p=800$ (so that $\kappa=0.2$), and $\gamma^2 = 5$. The
covariates follow an AR(1) model with $\rho= 0.5$ and 
$\Sigma_{jj} = 1$. 

\subsection{Estimating parameters}\label{sec:estimate-parameter}

We here explain how to estimate the signal strength $\gamma^2$ and
conditional variance $\tau_j^2$ needed to describe the distribution of
the LLR and MLE.

To estimate the signal strength, we use the \emph{ProbeFrontier}
method introduced in \cite{sur18}. As we have seen in Section
\ref{sec:exist-mle}, for each $\gamma$, there is a corresponding
problem dimension $\kappa(\gamma)$ on the phase transition curve, see
Figure \ref{fig:formula}: once $\kappa>\kappa(\gamma)$, the MLE no
longer exists asymptotically \cite{candes18}. The \emph{ProbeFrontier}
method searches for the smallest $\kappa$ such that the MLE ceases to
exist by sub-sampling observations. Once we obtain $\hat{\gamma}$, 
we set
$(\hat{\alpha},\hat{\sigma},\hat{\lambda})$ to be the solution to the
system of equations with parameters $(\kappa,\hat{\gamma})$. Because
the \emph{ProbeFrontier} method only checks whether the points are
separable, the quality of the estimate $\hat{\gamma}$ does not depend
upon whether the covariates are independent or not. We therefore
expect good performance across the board.

As to the conditional variance, since the covariates are Gaussian, it
can be estimated by a simple linear regression. Let
$\mbf{X}_{\bullet, -j}$ be the data matrix without the $j$th column,
and consider the residual sum of squares $\text{RSS}_j$ obtained by
regressing the $j$th column $\mbf{X}_{\bullet,j}$ onto
$\mbf{X}_{\bullet,-j}$. Then 
\[
\text{RSS}_j \,\, \sim {\tau_j^2}\,\, \chi^2_{n-p+1}. 
\] 
Hence,
\begin{equation}\label{eq:tau_hat}
    \hat{\tau}_j^2 = \frac{\text{RSS}_j / n}{1-\kappa}
\end{equation}
is nearly unbiased for $\tau_j^2$.\footnote{We also have
  $\text{RSS}_j = 1/\Theta_{jj}$, $\mbf{\Theta} = (\mbf{X}^\top \mbf{X})^{-1}$.} 
  
In our example, the covariates follow an AR(1) model and there is a
natural estimate of $\rho$ by maximum likelihood. This yields an
estimated covariance matrix $\hat{\mbf{\Sigma}}(\hat{\rho})$
parameterized by $\hat{\rho}$, which we then use to estimate the
conditional variance $\hat{\tau}^2_j(\hat{\rho})$. Below, we use both
the nonparametric estimates $\hat{\tau}_j$ and parametric estimates
$\hat{\tau}_j(\hat{\rho})$.

\subsection{Empirical performance of a $t$-test}\label{subsec:ttestest}

Imagine we want to use Theorem \ref{thm:finitemle} to calibrate a test
to decide whether $\beta_j=0$ or not. After plugging in estimated
parameters, a p-value for a two-sided test takes the form
\begin{equation}\label{eq:p_val_t}
    \hat{p}_j = 2\bar\Phi(\sqrt{n} \hat{\tau}_j|\hat{\beta}_j|/\hat{\sigma}),
\end{equation}
where $\bar\Phi(t) = \mathbb{P}(\mathcal{N}(0,1) > t)$.  In Table
\ref{tab:pval_t}, we report the proportion of p-values calculated from
\eqref{eq:p_val_t}, below some common cutoffs. To control type-I
errors, the proportion of p-values below 10\% should be at most about
10\% and similarly for any other level. The p-values computed from
true parameters show a correct behavior, as expected. If we use
estimated parameters, the p-values are also accurate and are as good
as those obtained from true parameters. In comparison, p-values from
classical theory are far from correct, as shown in Column 4. 

\begin{table}
\caption{Empirical performance of a $t$-test from
        \eqref{eq:p_val_t}. Each cell reports the p-value probability
      and its standard error (in parentheses) estimated over $B=10,000$
      repetitions. The first two columns use \emph{ProbeFrontier} to
      estimate the problem parameter $\hat{\sigma}$, and the two
      estimates of conditional variance from Section
      \ref{sec:estimate-parameter}. 
      The third column assumes knowledge of the signal-to-noise parameter
      $\gamma$.  The last column uses the normal approximation from R.}
\label{tab:pval_t}
\begin{tabular}{@{}p{3.5cm} cccc@{}}
   \hline
    & 1 & 2 & 3 &4\\
  &  $(\hat{\tau},\hat{\sigma})$ & $(\hat{\tau}(\hat{\rho}),\hat{\sigma})$   & $(\tau,\sigma_\star)$ & Classical \\
    \hline
    $\mathbb{P}(\text{P-value}\leq 10\%)$ & 10.09\% (0.30\%)&10.14\% (0.30\%) & 10.22\% (0.30\%) & 17.80\% (0.38\%)\\
    $\mathbb{P}(\text{P-value}\leq 5\%)$ &5.20\% (0.22\%) & 5.23\% (0.22\%) & 5.24\% (0.22\%) & 10.73\% (0.31\%)\\
    $\mathbb{P}(\text{P-value}\leq 1\%)$ &1.16\% (0.11\%)& 1.22\% (0.11\%) & 1.33\% (0.11\%) & 3.72\% (0.19\%)\\
    $\mathbb{P}(\text{P-value}\leq 0.5\%)$ &0.68\% (0.08\%) &0.70\% (0.08\%) & 0.74\% (0.08\%) & 2.43\% (0.15\%)\\
    \hline
 
\end{tabular}
\end{table}

\subsection{Coverage proportion}

We proceed to check whether the confidence intervals constructed from the
estimated parameters
\begin{equation}\label{eq:ciemp}
    \left[\frac{1}{\hat{\alpha}}\left(\hat{\beta}_j-\frac{\hat{\sigma}}{\sqrt{n}\hat{\tau}_j}z_{(1-\alpha/2)}\right),\frac{1}{\hat{\alpha}}\left(\hat{\beta}_j+\frac{\hat{\sigma}}{\sqrt{n}\hat{\tau}_j}z_{(1-\alpha/2)}\right)\right]
\end{equation}
achieve the desired coverage property. 

\begin{table}
\caption{Coverage proportion of a single variable. Each cell reports the proportion of times a variable $\beta_j$ is covered by the corresponding confidence interval from \eqref{eq:ciemp}, calculated over $B = 10,000$ repetitions; we chose the variable  to be  the same null coordinate as in Section \ref{subsec:ttestest}.  
    The standard errors are given between parentheses. The first two columns use estimated parameters, and the last one uses the true parameters.}
\label{tab:cov_prop_emp}
\begin{tabular}{@{}cccc@{}}
\hline
    Nominal coverage & 1& 2&3 \\
  $100(1-\alpha)$  & $(\hat{\tau},\hat{\sigma})$ &$(\hat{\tau}(\hat{\rho}),\hat{\sigma})$   & $(\tau,\sigma_\star)$ \\  
    \hline
    99.5 & 99.32 (0.08) & 99.30 (0.08) & 99.26 (0.09)\\
    99 & 98.84 (0.11) & 98.78 (0.11) & 98.67 (0.11) \\
    95 & 94.80 (0.22) & 94.77 (0.22) &94.76 (0.22) \\
    90 & 89.91 (0.30) & 89.86 (0.30) & 89.78 (0.30)\\
    \hline
\end{tabular}
\end{table}

\begin{table}
\caption{Proportion of variables inside the confidence
        intervals \eqref{eq:ciemp}. Each cell reports the proportion of \emph{all} the variables in each run falling within the corresponding confidence intervals from \eqref{eq:ciemp}, averaged
      over $B=10,000$ repetitions (standard errors in parentheses). The first two columns use estimated parameters, and the last one uses the true parameters.}
\label{tab:ci_est}
\begin{tabular}{@{}cccc@{}}
\hline
Nominal coverage & 1& 2& 3\\
  $100(1-\alpha)$  & $(\hat{\tau},\hat{\sigma})$ &$(\hat{\tau}(\hat{\rho}),\hat{\sigma})$   & $(\tau,\sigma_\star)$ \\
    \hline
    98 & 97.96 (0.01)&97.95 (0.01) & 97.85 (0.01)\\
    95 & 95.01 (0.01) &95.00 (0.01) & 94.85 (0.02)\\
    90 & 89.92 (0.02)& 89.91 (0.02) & 89.72 (0.02)\\
    80 & 79.99 (0.02) &79.99 (0.02) & 79.77 (0.03)\\
    \hline
\end{tabular}
\end{table}

We first test this in the context of Theorem \ref{thm:finitemle}, in particular \eqref{eq:nonnullsingle}. Table \ref{tab:cov_prop_emp} reports the proportion of times a single coordinate lies in the corresponding confidence interval from \eqref{eq:ciemp}.
 We observe that the coverage proportions are close to the respective targets, even with the estimated parameters. 

 Moving on, we study the accuracy of the estimated parameters in light
 of Theorem \ref{thm:bulk}.  This differs from our previous
 calculation: Table \ref{tab:cov_prop_emp} focuses on whether a single
 coordinate is covered, but now we compute the proportion of {\em all}
 the $p=800$ variables falling within the respective confidence
 intervals from \eqref{eq:ciemp}, in each single experiment. We report
 the mean of these proportions (Table \ref{tab:ci_est}), computed
 across 10,000 repetitions. Ideally, the proportion should be about
 the nominal coverage and this is what we observe.

\subsection{Empirical performance of the LRT} 

Lastly, we examine p-values for the LRT when the signal strength 
$\gamma^2$ is unknown. The p-values take the form
\begin{equation}\label{eq:p_val_lrt}
    \hat{p}_j = \prob\left(\chi^2_1\geq
      \frac{\hat{\lambda}}{\kappa\hat{\sigma}^2}\; 2\mathrm{LLR}_j\right)
\end{equation}
once we plug in estimated values for $\lambda_\star$ and
$\sigma_\star$. Table \ref{tab:pval_dev} displays the proportion of
p-values below some common cutoffs for the same null coordinate as in
Table \ref{tab:pval_t}.  Again, classical theory yields a gross
inflation of the proportion of p-values in the lower tail.  In
contrast, p-values from either estimated or true parameters display
the correct behavior.

\begin{table}
\caption{Empirical performance of the  LRT.  Each cell reports the p-value probability and its
      standard error (in parentheses) estimated over $B=10,000$
      repetitions. The first column uses \emph{ProbeFrontier}
      estimated factor $\hat{\lambda}/\kappa\hat{\sigma}^2$ whereas
    the second uses $\lambda_\star/\kappa\sigma_\star^2$. 
    The last column displays the results from classical theory.}
\label{tab:pval_dev}
\begin{tabular}{@{}p{2.5cm}c c c@{}}
 \hline
  & Estimated & True & Classical \\
    \hline
    $\mathbb{P}(\text{P-value}\leq 10\%)$ & 10.04\% (0.30\%)&10.06\% (0.30\%) & 17.86\% (0.38\%) \\
    $\mathbb{P}(\text{P-value}\leq 5\%)$ &5.19\% (0.22\%) & 5.25\% (0.22\%) &  10.76\% (0.31\%)\\
    $\mathbb{P}(\text{P-value}\leq 1\%)$ &1.17 \% (0.11\%)& 1.18\% (0.11\%) & 3.75\% (0.19\%)\\
    $\mathbb{P}(\text{P-value}\leq 0.5\%)$ &0.68\% (0.08\%) &0.69\% (0.08\%) & 2.49\% (0.15\%)\\
    \hline
\end{tabular}
\end{table}

\section{A sub-Gaussian example} \label{sec:sub-gaussian}

Our model assumes that the covariates arise from a multivariate normal
distribution. As in \cite[Section 4.g]{sur18}, however, we expect that
our results apply to a broad class of covariate distributions, in
particular, when they have sufficiently light tails. To test this, we
consider a logistic regression problem with covariates drawn from a
sub-Gaussian distribution that is inspired by genetic studies, and
examine the accuracy of null p-values and confidence intervals
proposed in this paper.

Since the signal strength $\gamma^2$ and conditional variances
$\tau_j^2$ are unknown in practice, we use throughout the
\emph{ProbeFrontier} method\footnote{Here, we resample 10 times for
  each $\kappa$.} and \eqref{eq:tau_hat} to obtain accurate estimates.

\subsection{Model setting}

In genome-wide association studies (GWAS), one often wishes to
determine how a binary response $Y$ depends on single nucleotide
polymorphisms (SNPs); here, each sample of the covariates measures the
genotype of a collection of SNPs, and typically takes on values in
$\{0,1,2 \}^p$. Because neighboring SNPs are usually correlated, GWAS
inspired datasets form an excellent platform for testing our theory.

Hidden Markov Models (HMMs) are a broad class of distributions that
have been widely used to characterize the behavior of SNPs
\cite{sesia18,rastas05,stephens06,kimmel05}.  Here, we study the
applicability of our theory when the covariates are sampled from a
class of HMMs, and consider the specific model implemented in the
fastPHASE software (see \cite[Section 5]{sesia18} for details) that
can be parametrized by three vectors
$(\mbf{r}, \mbf{\eta}, \mbf{\theta})$. We generate $n = 5000$
independent observations $(\mbf{X}_i, y_i)_{1 \leq i \leq n}$ by first
sampling $\mbf{X}_i$ from an HMM with parameters
$\mbf{r}=\mbf{r}_0,\mbf{\eta}=\mbf{\eta}_0,
\mbf{\theta}=\mbf{\theta}_0$ and $p=1454 $, so that $\kappa =0.29$,
and then sampling
$y_i \sim \mathrm{Ber}(\sigma(\mbf{X}_i^{\top}\mbf{\beta}))$. 
The \verb+SNPknock+ package \cite{snpknock} was used for sampling the
covariates and the parameter values are available at
\url{https://github.com/zq00/logisticMLE}.  We then standardize
the design matrix so that each column has zero mean and unit norm.
The regression coefficients are obtained as follows: we randomly pick
100 coordinates to be i.i.d.~draws from a mean zero normal
distribution with standard deviation 10, and the remaining coordinates
vanish. We repeat this experiment $B=5000$ times.

\subsection{Accuracy of null p-values}
We focus on a single null coordinate and, across the $B$ replicates, calculate p-values based on four test statistics---(a) the classical $t$-test,  which yields the p-value formula  $2\bar{\Phi}(\sqrt{n}|\hat{\beta}_j|/\hat{\sigma}_j)$; here $\hat{\sigma}_j$ is taken to be the estimate of the standard error from \verb+R+, (b) the classical LRT, (c) the $t$-test suggested by Theorem \ref{thm:finitemle}; in this case, the formula is the same as in (a), except that $\hat{\sigma}_j = \hat{\sigma}/\hat{\tau}_j$, where $\hat{\sigma}$ is estimated from \emph{ProbeFrontier} and $\hat{\tau}_j$  from \eqref{eq:tau_hat}, and finally, (d) the LRT based on Theorem \ref{thm:llr}; here again, the rescaling constant is specified via the estimates $\hat{\sigma}$, $\hat{\lambda}$ produced by \emph{ProbeFrontier}. The histograms of the classical p-values are shown in Figures \ref{fig:ex_tr} and  \ref{fig:ex_llr_unadj}---these are far from the uniform distribution, with severe inflation near the lower tail. The histograms of the two sets of p-values based on our theory are displayed in Figures \ref{fig:ex_t_hist} and \ref{fig:ex_llr_hist}, whereas the corresponding empirical cdfs can be seen in Figures \ref{fig:ex_t_ecdf} and \ref{fig:ex_llr_ecdf}. In both of these cases, we observe a remarkable proximity to the uniform distribution. Furthermore, Table \ref{tab:ex_pval} reports the proportion of null p-values below a collection of thresholds; both the $t$-test and the LRT suggested by our results provide accurate control of the type-I error. These empirical observations indicate that our theory likely applies to a much broader class of non-Gaussian distributions. 

\begin{figure*}
    \begin{subfigure}[t]{0.3\textwidth}
        \centering
        \includegraphics[width=1\textwidth]{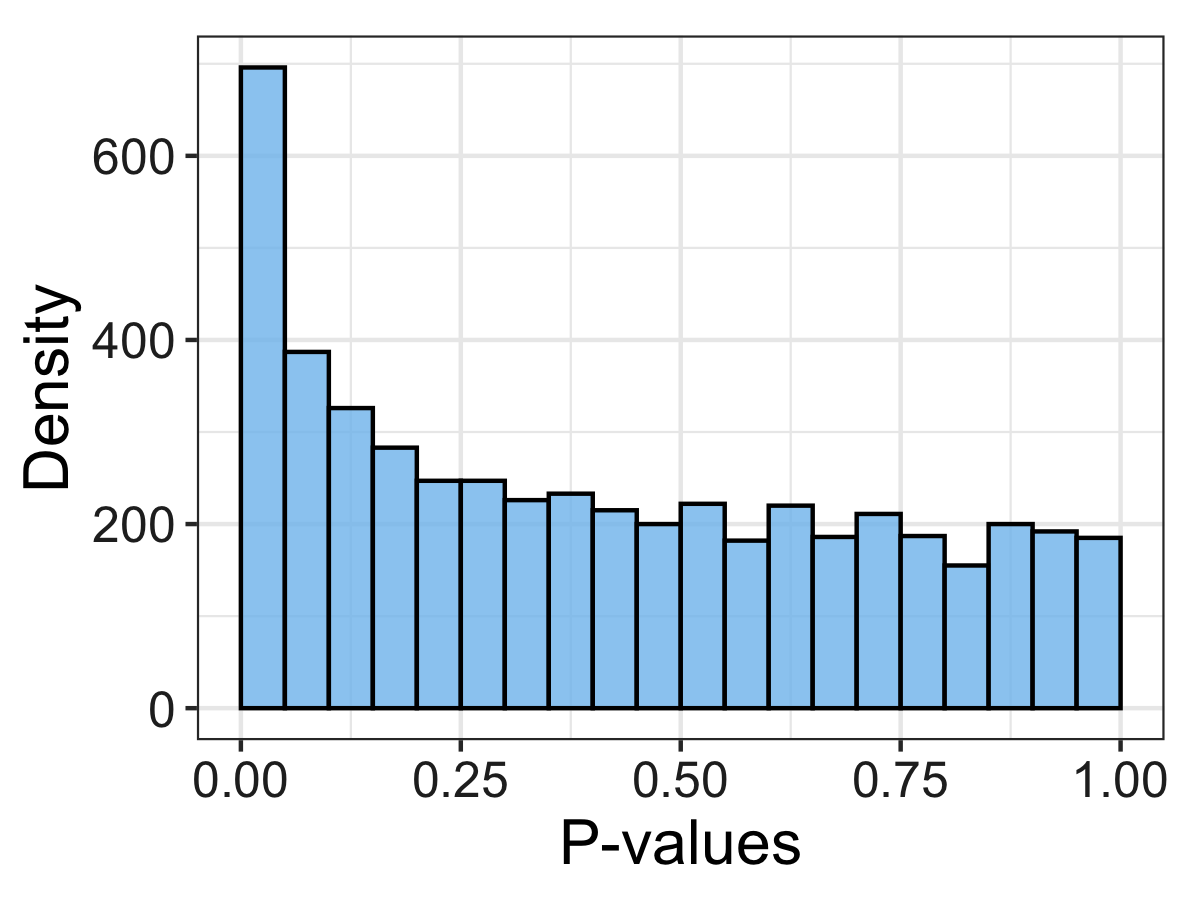}
        \caption{}
        \label{fig:ex_tr}
    \end{subfigure}
    \hspace*{0em}
    \begin{subfigure}[t]{0.3\textwidth}
        \centering
        \includegraphics[width=1\textwidth]{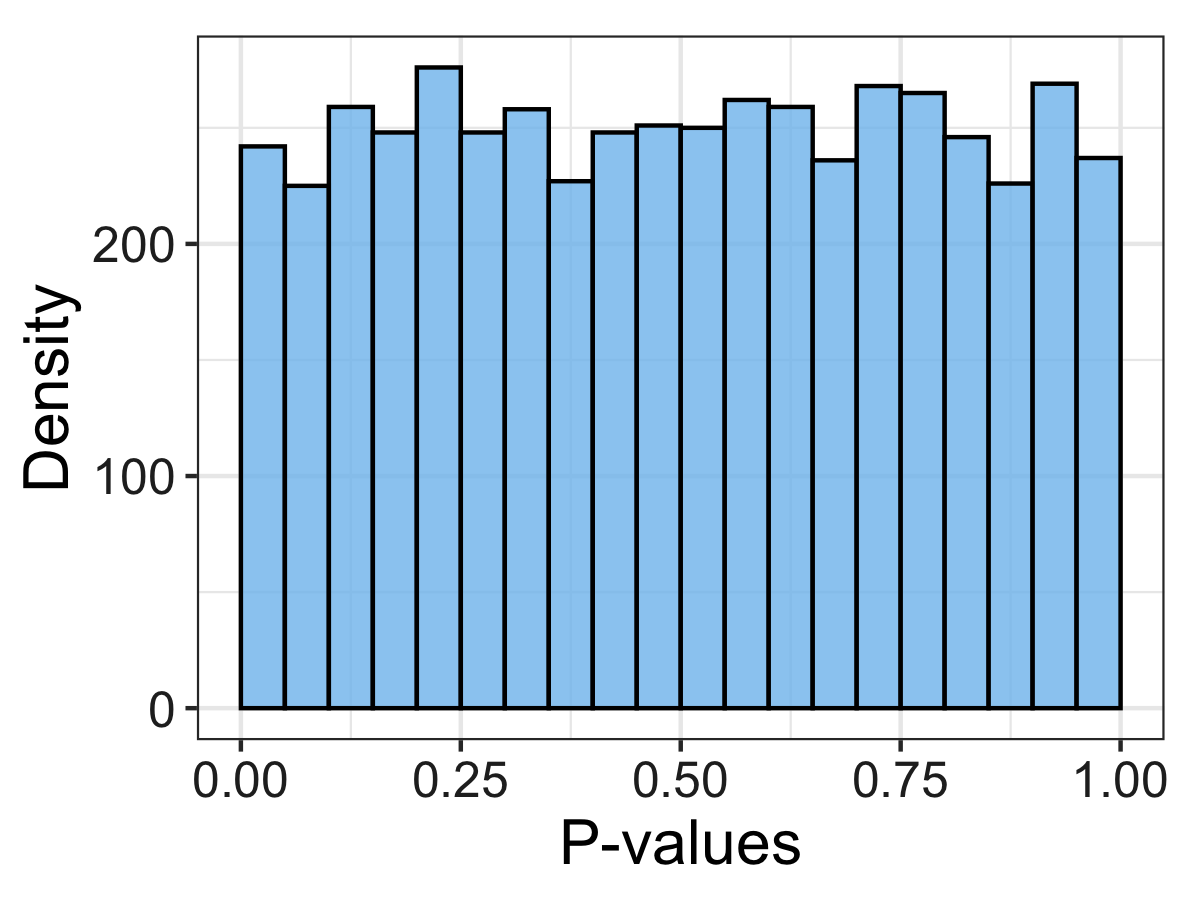}
        \caption{}
        \label{fig:ex_t_hist}
    \end{subfigure}%
    \centering
    \begin{subfigure}[t]{0.3\textwidth}
        \centering
        \includegraphics[width=1\textwidth]{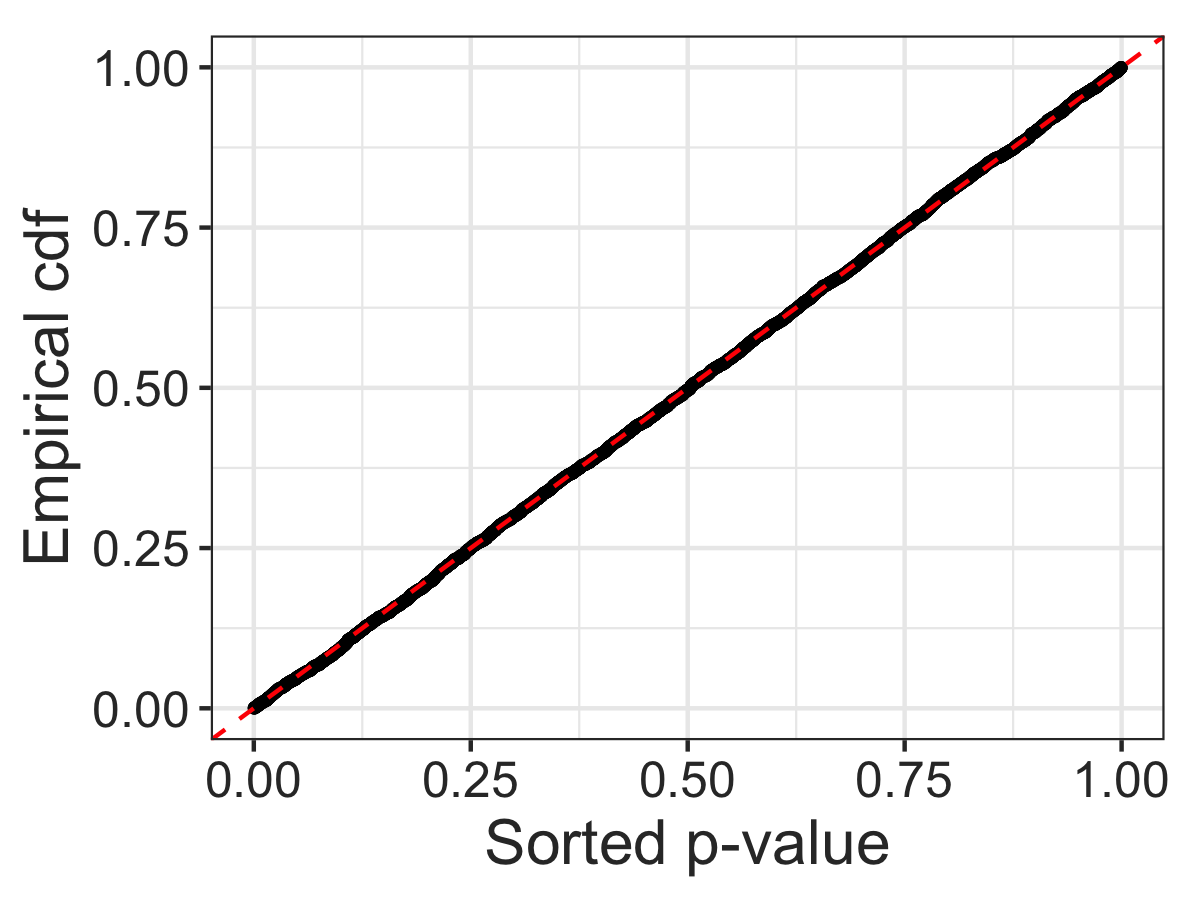}
        \caption{}
        \label{fig:ex_t_ecdf}
    \end{subfigure}%
    \hspace*{0em}
    \caption{Distribution of null p-values from a two-sided
        $t$-test. Histograms of p-values are calculated by
      $p_j=2\bar{\Phi}(\sqrt{n}|\hat{\beta}_j|/\hat{\sigma}_j)$. (a)
      $\hat{\sigma}_j$ is taken to be the standard error from R. (b)
      $\hat{\sigma}_j=\hat{\sigma}/\hat{\tau}_j$, where $\hat{\sigma}$ is
      estimated by \emph{ProbeFrontier} and $\hat{\tau}_j$ is from
      \eqref{eq:tau_hat} (c) Empirical cdf of the p-values in (b).}
    \label{fig:pt}
\end{figure*}

\begin{figure*}
    \begin{subfigure}[t]{0.3\textwidth}
        \centering
        \includegraphics[width=1\textwidth]{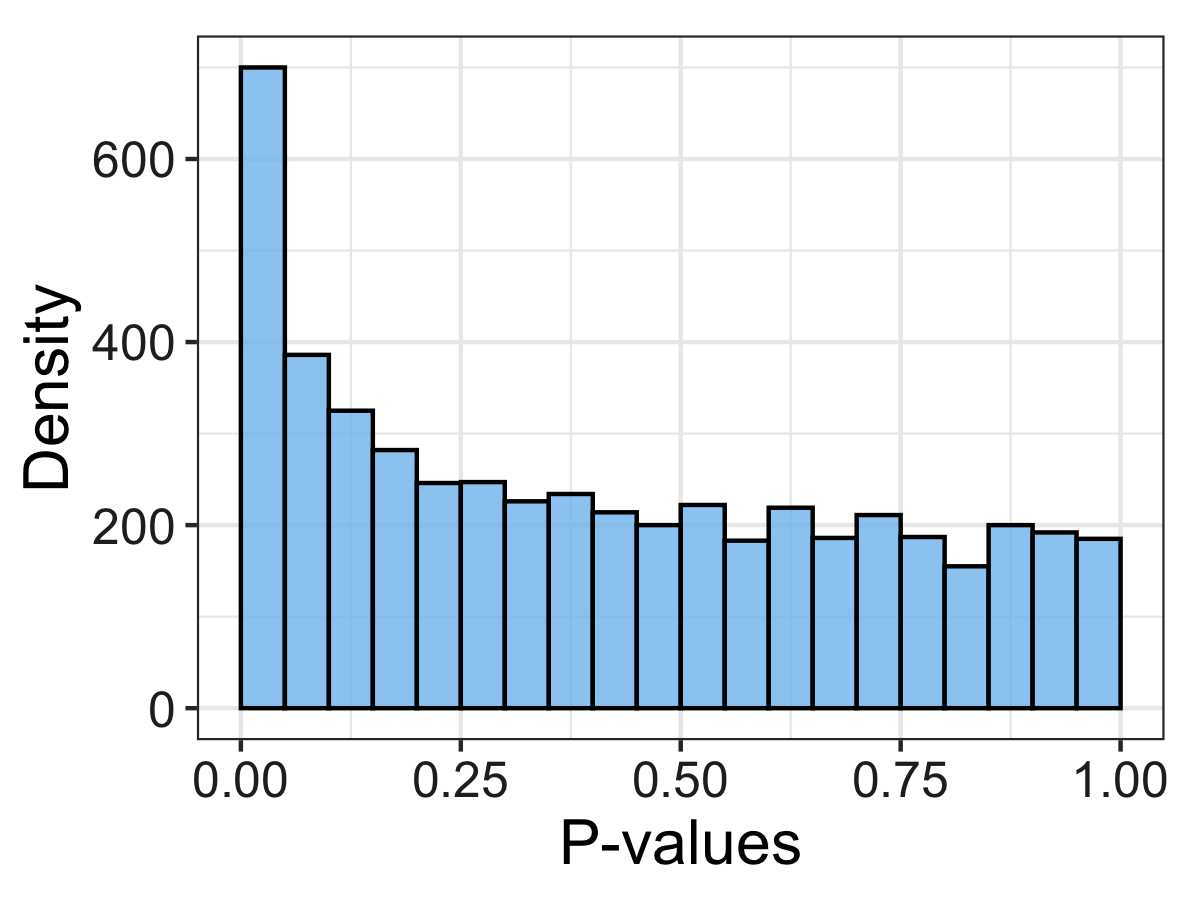}
        \caption{}
        \label{fig:ex_llr_unadj}
    \end{subfigure}
    \hspace*{0em}
    \begin{subfigure}[t]{0.3\textwidth}
        \centering
        \includegraphics[width=1\textwidth]{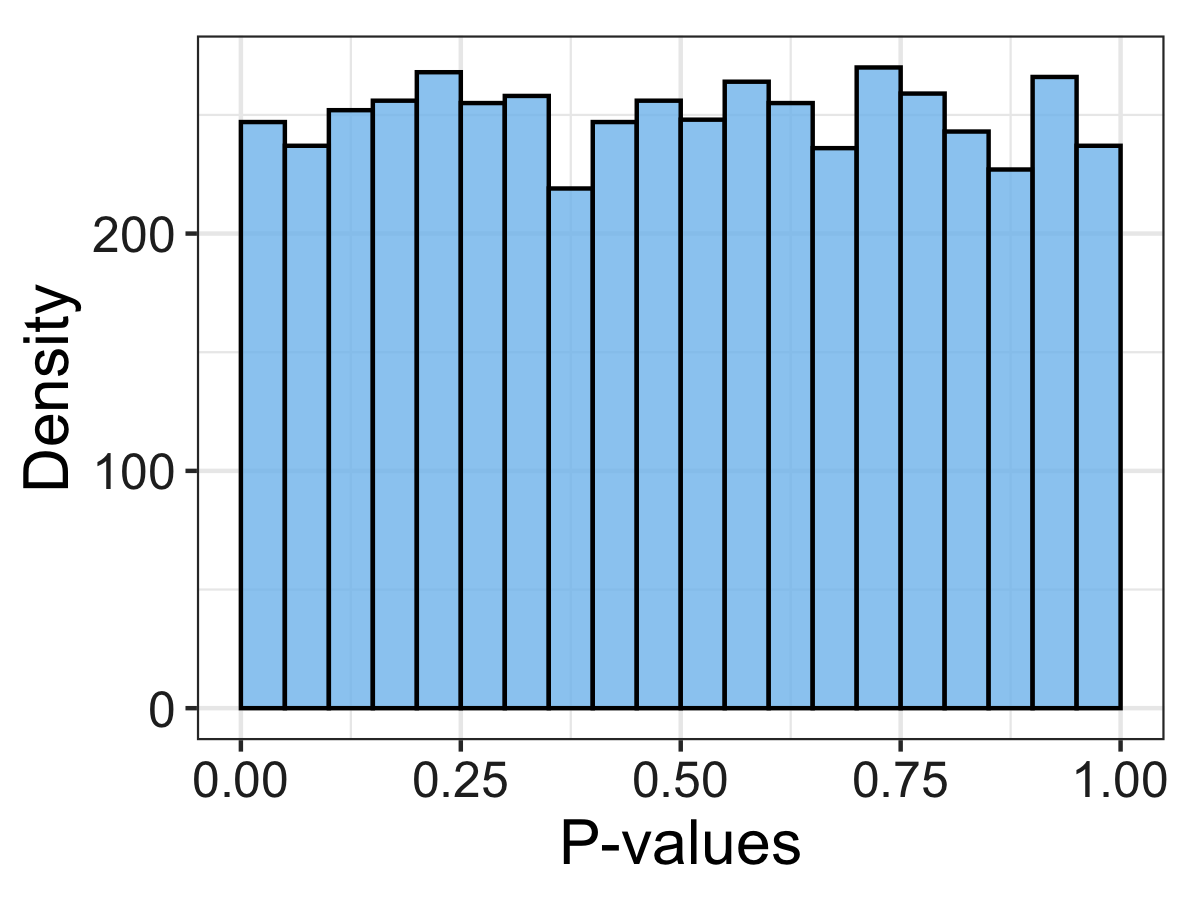}
        \caption{}
        \label{fig:ex_llr_hist}
    \end{subfigure}%
    \centering
    \begin{subfigure}[t]{0.3\textwidth}
        \centering
        \includegraphics[width=1\textwidth]{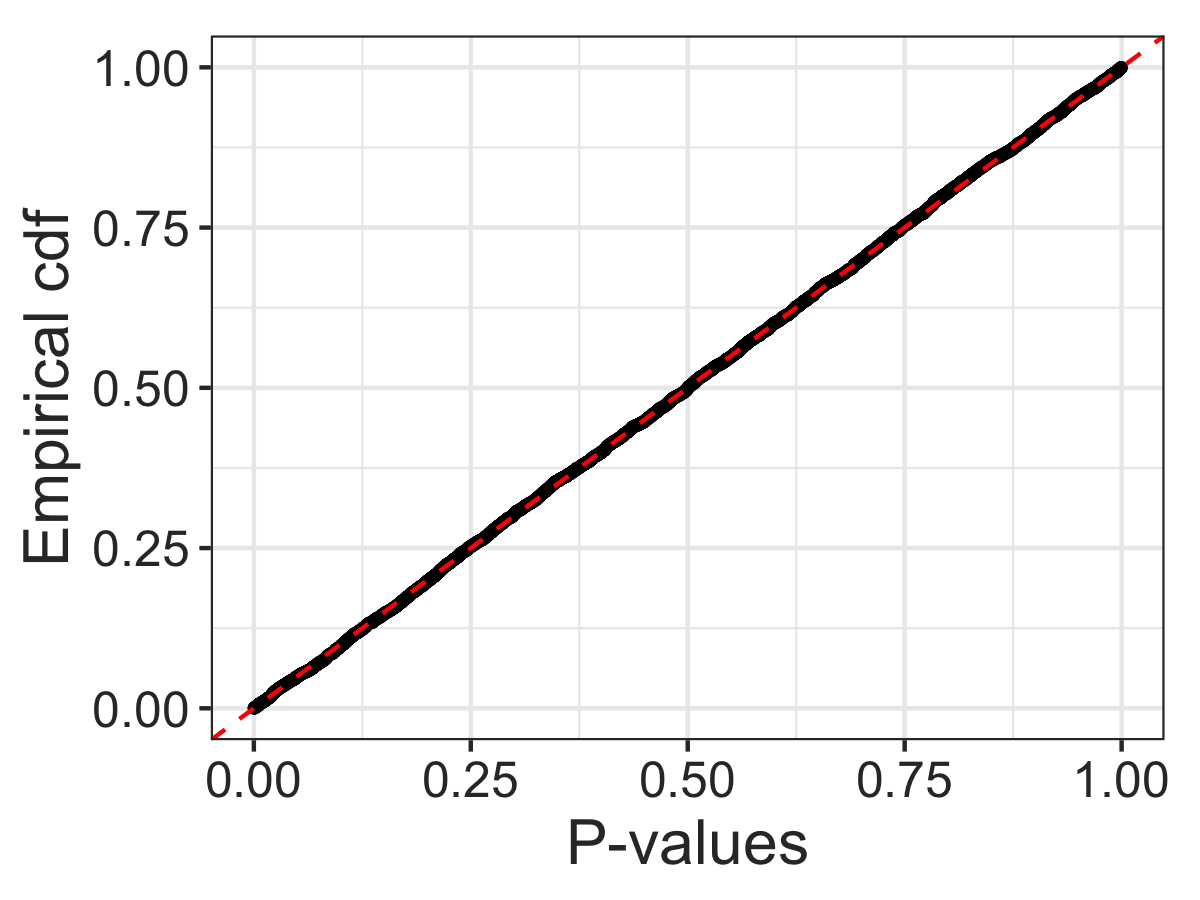}
        \caption{}
        \label{fig:ex_llr_ecdf}
    \end{subfigure}%
    \hspace*{0em}
    \caption{Distribution of null p-values calculated from the LRT (a) Histogram of p-values based on the chi-squared distribution (with 1 degree of freedom). (b) Histogram of p-values based on the re-scaled chi-squared distribution; the re-scaling factor is estimated by \emph{ProbeFrontier}. (c) Empirical cdf of p-values from (b).}
    \label{fig:p_lrt}
\end{figure*}

\begin{table}
\caption{Empirical performance of testing a null. Each
      cell reports the p-value probability and its standard error (in
      parentheses) estimated over $B=5,000$ repetitions. The p-values
      are calculated from a two sided $t$-test (as in Figure
      \ref{fig:ex_t_ecdf}) and the LRT (as in Figure \ref{fig:ex_llr_ecdf}). }
\label{tab:ex_pval}
\begin{tabular}{@{}p{2.5 cm}p{2cm} p{2cm}@{}}
\hline
    & $t$-test & LRT \\
    \hline
    $\mathbb{P}(\text{P-value}\leq 10\%)$ & 9.34\% (0.41\%)&9.68\% (0.41\%)\\
    $\mathbb{P}(\text{P-value}\leq 5\%)$ & 4.84\% (0.30\%) &4.94\% (0.30\%)\\
    $\mathbb{P}(\text{P-value}\leq 1\%)$ & 0.96\% (0.14\%)& 0.94\% (0.14\%)\\
    $\mathbb{P}(\text{P-value}\leq 0.1\%)$ &0.08\% (0.04\%) &0.08\% (0.04\%)\\
    \hline
\end{tabular}
\end{table}

\subsection{Coverage proportion}
We proceed to check the accuracy of the confidence intervals  described by \eqref{eq:ciemp}. We consider a single coordinate $\beta_j$ (we chose $\beta_j \neq 0$) and report the proportion of times  \eqref{eq:ciemp} covers $\beta_j$ across the $B$ repetitions (Table \ref{tab:cov_prop_nonnull_subgaussian}). At each level, the  empirical coverage proportion agrees with the desired target level, validating the marginal distribution \eqref{eq:nonnullsingle} in non-Gaussian settings. To investigate the efficacy of \eqref{eq:nonnullsingle} further, we calculate the standardized versions of the MLE given by 
\begin{equation}\label{eq:T_emp}
\hat{\zt}_j = \frac{\sqrt{n}(\hat{\beta}_j-\hat{\alpha} \beta_j)}{\hat{\sigma}/\hat{\tau}_j}
\end{equation}
for each run of the experiment; recall that the estimates $\hat{\alpha},\hat{\sigma}, \hat{\tau}_j$ arise from the \emph{ProbeFrontier} method and \eqref{eq:tau_hat}. 
 Figure \ref{fig:qqplot_subgaussian}  displays a qqplot of the empirical quantiles of $\hat{\zt}_j$ versus the standard normal quantiles, and once again, we observe a remarkable agreement. 

\begin{table}
\caption{Coverage proportion of a single variable. Each
      cell reports the proportion of times $\beta_j$ falls within
      \eqref{eq:ciemp}, calculated over $B=5,000$ repetitions; the
      standard errors are provided as well. The unknown signal strength $\gamma^2$ is estimated by \emph{ProbeFrontier}. }
\label{tab:cov_prop_nonnull_subgaussian}
\begin{tabular}{@{}cccccc@{}}
\hline
     Nominal coverage & & & & & \\
$100(1-\alpha)$ & 99 & 98 & 95 & 90 & 80 \\
        \hline
        Empirical coverage & 99.04 & 97.98 & 94.98 & 89.9 & 80.88 \\
        Standard error & 0.2 & 0.2 & 0.3 & 0.4 & 0.6 \\
      \hline
\end{tabular}
\end{table}

\begin{figure}
\begin{center}
          \includegraphics[width=6.5cm,keepaspectratio]{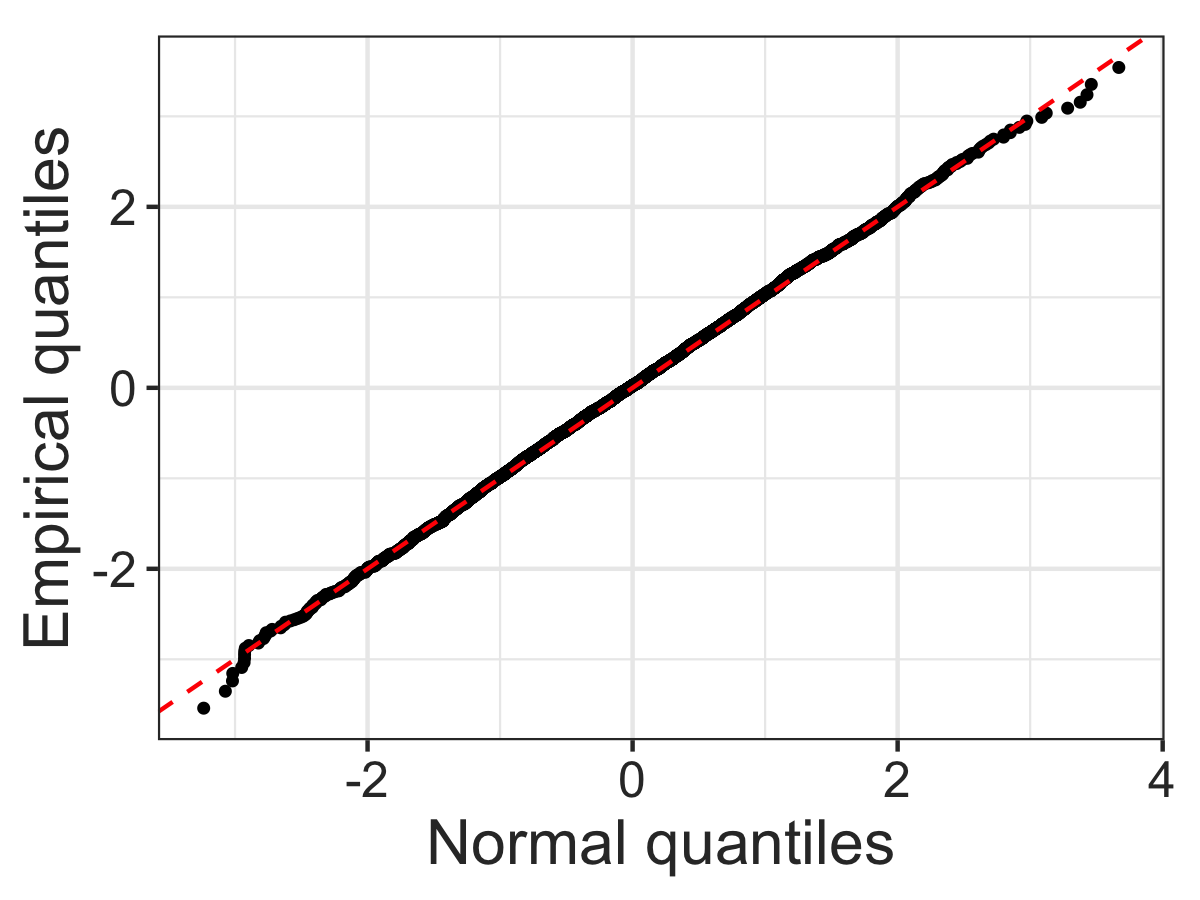}
\end{center}
\caption{ Quantiles of the empirical distribution of the MLE coordinate from Table \ref{tab:cov_prop_nonnull_subgaussian}, standardized as in \eqref{eq:T_emp}, versus standard normal quantiles. }
\label{fig:qqplot_subgaussian}
\end{figure}

Finally, we turn to study the performance of Theorem \ref{thm:bulk}. In each experiment, we compute the proportion of variables covered by the corresponding intervals in \eqref{eq:ciemp} and report the mean (Table \ref{tab:ex_ci}) across the $5000$ replicates. Once again, the average coverages  remained close to the desired thresholds for all the levels considered, demonstrating the applicability of the bulk result \eqref{eq:bulk_dep} beyond the setting of Gaussian covariates. 

\begin{table}
\caption{ Proportion of variables inside the confidence
        intervals \eqref{eq:ciemp}. Each cell reports the average
      coverage estimated over $B=5,000$ repetitions. The standard errors are shown as well.}
\label{tab:ex_ci}
\begin{tabular}{@{}cccccc@{}}
\hline
     Nominal coverage & & & &  \\
$100(1-\alpha)$ & 98 & 95 & 90 & 80  \\
        \hline
        Empirical coverage & 98.03 & 95.14& 90.1 & 80.2  \\
        Standard error & 0.01 & 0.01 & 0.02 & 0.02  \\
      \hline
\end{tabular}
\end{table}

\section{Models with an intercept}\label{sec:intercept_method}
In this section, we study the asymptotic distribution of a logistic MLE when the model contains an intercept, i.e. the likelihood of $y_i$ conditional on the  covariates $\mbf{x}_i$ is given by
\begin{equation}\label{eq:logistic_model_intercept}
    \prob(y_i = 1\,|\, \mbf{x}_i) = 1/ (1+\exp(-y_i (\beta_0 + \mbf{x}_i^\top \mbf{\beta}))).
\end{equation}
Suppose the intercept $\beta_0=o(1)$, then all the earlier theorems
about the distribution of $\hat{\beta}_j$ and $\llr_j$ apply as long
as $j \ge 1$; the parameters in Theorem \ref{thm:finitemle} and
Theorem \ref{thm:llr} are still solutions to
\eqref{eq:system_equation}. On the other hand, we conjecture that when
$\beta_0$ is not asymptotically negligible, the MLE remains
asymptotically Gaussian as in Theorem \ref{thm:finitemle}, but the
parameters $\alpha_\star$ $\sigma_\star$ and $\lambda_\star$ now
depend on $\beta_0$ as well as $\kappa$ and $\gamma$. (The phase
transition curve in \cite[Theorem~2.1]{candes18} also depends on both
$\gamma$ and $\beta_0$.) We conduct simulation studies to verify our
conjecture (Section \ref{sec:theory_intercept_simulation}) and discuss
how to estimate these parameters when the intercept is unknown
(Section \ref{sec:estimate_param_intercept}).
\subsection{Asymptotic distribution of the MLE}
Before we describe our conjecture when $\beta_0\neq o(1)$, we note
that Proposition \ref{prop:indtodep} and Lemma \ref{lemma:exact}
applied to $\{\hat \beta_j\}_{j \ge 1}$---all coordinates except the
intercept term---still hold because the rotation invariance argument
operates in the same way.  Thus, to establish the asymptotic MLE
distribution, we only need to study the limit of $\alpha(n)$ and
$\sigma(n)$ defined in Eqn. \eqref{eq:alphansigmandef}. We conjecture
that, as $n,p\to\infty$ while $p/n\to\kappa$, $\alpha(n)$ and
$\sigma(n)$ approach limits that can be determined by solving a system
of four equations.  
\begin{conjecture} \label{thm:theory_intercept}
Consider the logistic regression model \eqref{eq:logistic_model_intercept} and assume that $(\kappa, \gamma,\beta_0)$ is such that the MLE exists asymptotically. Denote $\hat{\beta}_0$ to be the MLE of the intercept and define $\alpha(n)$, $\sigma(n)$ as in Eqn. \eqref{eq:alphansigmandef}. Then, as $n,p\to\infty$ and $p/n\to \kappa > 0$,  
\begin{equation}\label{eq:asymp_param_intercept}
    \hat{\beta}_0 \stackrel{P}{\longrightarrow} b_\star,\quad  
    \alpha(n) \stackrel{P}{\longrightarrow} \alpha_\star,  \quad 
    \sigma(n) \stackrel{P}{\longrightarrow} \sigma_\star,
\end{equation}
and Theorem \ref{thm:finitemle} and Theorem \ref{thm:llr} hold with the set of parameters in \eqref{eq:asymp_param_intercept}. In \eqref{eq:asymp_param_intercept}, $\alpha_\star, \sigma_\star, b_\star$ are such that together with another constant  $\lambda_{\star}$, these solve a system of equations in four variables $(\alpha, \sigma,\lambda, b)$ given by
 \begin{equation}\label{eq:system_equation_intercept}
  \begin{dcases}
&\kappa^2 \sigma^2  =  \E{\rho'(-S_1) \left(\lambda \rho'(\prox_{\lambda \rho}(S_2))\right)^2 +\rho'(S_1)\left(\lambda \rho'(\prox_{\lambda \rho}(-S_2))\right)^2} \\
& 1-\kappa  = \E{\frac{\rho'(-S_1)}{1+\lambda \rho''(\prox_{\lambda \rho}(S_2))}+ \frac{\rho'(S_1)}{1+\lambda \rho''(\prox_{\lambda \rho}(-S_2))}}  \\
& 0 = \E{\rho'(\prox_{\lambda \rho}(S_2)) \rho'(-S_1)S_1 - \rho'(\prox_{\lambda \rho}(-S_2)) \rho'(S_1)S_1} \\
& 0 = \E{-\rho'(\prox_{\lambda \rho}(-S_2))\rho'(S_1) + \rho'(\prox_{\lambda \rho}(S_2))\rho'(-S_1)}.
\end{dcases}
\end{equation}
where $(Z_1, Z_2)\sim\N(0, \mbf{I}_2)$ and 
\begin{equation}\label{eq:defn_S}
S_1 = \gamma Z_1 + \beta_0 ,\quad S_2 = \alpha \gamma Z_1 + \sigma\sqrt{\kappa} Z_2 + b. 
\end{equation}
\end{conjecture}
$S_1$ and $S_2$ defined in \eqref{eq:defn_S} are related to $Q_1$ and
$Q_2$ in Eqn. \eqref{eq:system_equation} as $S_1 = -Q_1 + \beta_0$ and
$S_2 = Q_2 + b$. Compared to Eqn. \eqref{eq:system_equation},
Eqn. \eqref{eq:system_equation_intercept} has four equations, and the
fourth equation characterizes the limit of the estimated
intercept. When $\beta_0 = 0$, the set of equations \eqref{eq:system_equation_intercept} reduces to the system of equations \eqref{eq:system_equation}. 

In sum, Conjecture \ref{thm:theory_intercept} states that the marginal
distribution of a logistic MLE $\hat{\beta}_j$ is asymptotically
Gaussian with mean $\alpha_\star \beta_j$ and variance
$\sigma_\star^2/\tau_j^2$, where the parameters $\alpha_\star$ and
$\sigma_\star$ are determined by Eqn.
\eqref{eq:system_equation_intercept}.

\subsubsection{Finite sample accuracy}\label{sec:theory_intercept_simulation}

We study the accuracy of our conjecture through simulated examples,
where we fix $\beta_0 = 1$, but otherwise use the same setting as in
Section \ref{sec:bulk}. First, we report the coverage probability of a
single non-null variable on using the confidence interval from Eqn.
\eqref{eq:cisingle} (Table
\ref{tab:cov_prop_nonnull_intercept}). Although the confidence
intervals slightly undercovers the true coefficient, they are
reasonably accurate as the error is within 0.5\%. We also report the
coverage proportion of {\emph{all}} of the variables (Table
\ref{tab:bulk_cov_intercept}) which shows that the $(1-\alpha)$
confidence interval covers approximately $(1-\alpha)$ of all of the
variables in a single-shot experiment. We report results for different
covariance matrices in the supplementary material \cite[Section~D]{zhao2021_supp}, and observe that the
performance is consistent across different types of matrices. Finally,
we compute the adjusted p-values for a likelihood ratio statistics and
report the distributions of p-values (Table
\ref{tab:pval_lrt_intercept}). The adjusted p-values achieve the
desired type I error because the proportion of p-values below each 
level is as we expect.
\begin{table*}
\caption{Coverage proportion of a single non-null variable when the logistic model includes an intercept $\beta_0 = 1$ and the covariance matrix is defined as $\Sigma_{i,j} = 0.5^{|i-j|}$. Each cell reports the proportion of times $\beta_j$ falls within the adjusted $(1-\alpha)$ CI using theoretical (Column I) and estimated parameters (Column II). The proportions are calculated over $B=100,000$ repetitions in Column I and $B=10,000$ in Column II, and the standard errors are provided as well. }
\label{tab:cov_prop_nonnull_intercept}
\begin{tabular}{@{}ccccc@{}}
\hline
& \multicolumn{2}{c}{I. Theoretical } & \multicolumn{2}{c}{II. Estimated} \\
Nominal coverage & Empirical & Standard & Empirical & Standard  \\
$100(1-\alpha)$ & coverage & error & coverage & error \\
\hline
 99    & 98.96 & 0.03 & 98.74 & 0.11\\
 98    & 97.90 & 0.05 & 97.71 & 0.15 \\
 95    & 94.86 & 0.07 & 94.56 & 0.23\\
 90    & 89.83 & 0.10 & 89.18 & 0.31 \\
 80    & 79.88 & 0.13 & 79.34 & 0.40 \\
 \hline
\end{tabular}
\end{table*}

\begin{table*}
  \caption{Each cell reports the proportion of \emph{all} the
    variables in each run that fall within the corresponding intervals
    from \eqref{eq:cisingle} when $\beta_0 = 1$ in the same experiment
    as for Table \ref{tab:cov_prop_nonnull_intercept}. The standard
    errors are provided as well.}
\label{tab:bulk_cov_intercept}
\begin{tabular}{@{}ccccc@{}}
\hline
& \multicolumn{2}{c}{I. Theoretical} & \multicolumn{2}{c}{II. Estimated} \\
Nominal coverage & Empirical & Standard & Empirical & Standard  \\
$100(1-\alpha)$ & coverage & error & coverage & error \\
\hline
 99    & 98.897 & 0.002 & 98.79 & 0.11 \\
 98    & 97.858 & 0.003 & 97.73 & 0.11 \\
 95    & 94.811 & 0.005 & 94.64 & 0.12  \\
 90    & 89.790 & 0.008 & 89.55 & 0.12 \\
 80    & 79.808 & 0.010 & 79.49 & 0.12 \\
 \hline
\end{tabular}
\end{table*}

\begin{table}
\caption{Empirical performance of the LRT using adjusted p-values when the model has an intercept $\beta_0 = 1$.  Each cell reports the p-value probability for a random null coordinate estimated in $B=10,000$
      repetitions and the standard errors are provided in parentheses. Column I uses theoretical parameters, column II uses estimated parameters and Column III uses classical theory without adjustment.}
\label{tab:pval_lrt_intercept}
\begin{tabular}{@{}p{2.5cm}c c c@{}}
 \hline
   & I. Theoretical & II. Estimated & III. Classical \\
    \hline
    $\mathbb{P}(\text{P-value}\leq 10\%)$
    & 9.98 (0.30) & 10.04 (0.30) & 18.93 (0.45) \\
    $\mathbb{P}(\text{P-value}\leq 5\%)$
    & 4.92 (0.21) & 5.02 (0.22) & 11.62 (0.41) \\
    $\mathbb{P}(\text{P-value}\leq 1\%)$
    & 0.90 (0.09) & 0.99 (0.10) & 3.78 (0.29)  \\
    $\mathbb{P}(\text{P-value}\leq 0.5\%)$
    & 0.54 (0.07) & 0.61 (0.08) & 2.46 (0.25) \\
    \hline
\end{tabular}
\end{table}

\subsubsection{Effect of the intercept}
We study the effect of $\beta_0$ on the parameters $\alpha_\star$ and
$\sigma_\star$ by showing the theoretical predictions at
$\gamma^2 = 5$ for different choices of $\beta_0$ (Table
\ref{tab:theoretical_params}). We observe that all of the parameters
increase as $\beta_0$ increases (Column I). We should thus not ignore
the intercept when it is not trivially small. Because the intercept
term is not Gaussian, a model with an explicit intercept term is not
equivalent to a model without an explicit intercept term and a
matching overall signal strength.  As a demonstration, we compare the
parameters obtained here with solutions to the system of three
equations \eqref{eq:system_equation} when we merge the intercept with
the other coefficients, i.e. setting $\gamma = \sqrt{5 +
  \beta_0^2}$. This approximation is accurate unless $\beta_0$ is
large, for example, $\beta_0 = 2$ as in Table
\ref{tab:theoretical_params}, row 4.  

\begin{table}
      \caption{The theoretical parameters for different $\beta_0$ when $\gamma^2 = 5$. Column I shows the solutions to the system of four equations \eqref{eq:system_equation_intercept} while Column II shows the solution to the system of three equations \eqref{eq:system_equation} assuming $\beta_0 = 0$ and $\gamma = \sqrt{5 + \beta_0^2}$.}
    \label{tab:theoretical_params}
    \begin{tabular}{@{}c c c c c  c c c c@{}} 
    \hline
    & \multicolumn{4}{c}{I. Theoretical parameters } & \multicolumn{4}{c}{II. Merging intercept} \\
       $\beta_0$  &$\alpha_\star$ & $\sigma_\star$ & $\lambda_\star$ & $b_\star$ & $\sqrt{5 + \beta_0^2}$  & $\alpha_\star$ & $\sigma_\star$ & $\lambda_\star$\\
       \hline 
       0  & 1.50 & 4.75 & 3.03 & 0 & 2.24 & 1.50 & 4.75 & 3.03 \\
       0.5 & 1.51 & 4.84 & 3.13 & 0.76 & 2.29 & 1.51 & 4.84 & 3.12 \\
       1 & 1.56 & 5.16 & 3.45 & 1.559 & 2.45 & 1.55 & 5.13 & 3.42 \\
       2 & 1.83 & 7.01 & 5.47 & 3.68 & 3.00 & 1.75 & 6.45 & 4.83 \\
       2.5 & 2.31 & 10.00 & 8.96 & 5.80 & 3.35 & 1.95 & 7.73 & 6.26\\ 
       \hline
    \end{tabular}
\end{table}

\subsection{Estimating model parameters}\label{sec:estimate_param_intercept}

Conjecture \ref{thm:theory_intercept} suggests that the MLE
distribution for logistic models with intercepts is determined by
$\kappa$, $\gamma$ and $\beta_0$. In this section, we introduce a
procedure to estimate the unknown $\gamma$ and $\beta_0$ based on two
observable quantities. First, the phase transition curve
\cite{candes18} determines a problem dimension
$\kappa_s = h(\beta_0, \gamma)$ such that if $\kappa > \kappa_s$, then the MLE does not exist asymptotically almost surely.
We use the \emph{ProbeFrontier} (see Section
\ref{sec:estimate-parameter}) method to estimate $\kappa_s$. In turn,
the estimated $\hat{\kappa}_s$ provides the estimating equation
\begin{equation}\label{eq:EstimateIntercept_Eq1}
    h(\beta_0, \gamma) = \hat{\kappa}_s.
\end{equation}
Second, the marginal probability $p$ of observing a positive outcome is determined by $\beta_0$ and $\gamma$ since
\[
p = \prob(Y = 1)= \E{ 1/{1+\exp(-\beta_0 - \gamma Z)}}, \quad Z\sim\N(0,1).
\]
Here, we substitute a Gaussian variable for $X^\top \beta$, since
$X^\top \beta \sim \N(0,\gamma^2)$. We therefore use the observed
proportion of positive outcomes $\hat p$ to get a second estimating
equation
\begin{equation}\label{eq:EstimateIntercept_Eq2}
    \E{ 1/{1+\exp(-\beta_0 - \gamma Z)}} = \hat{p}.
\end{equation}
Solving \eqref{eq:EstimateIntercept_Eq1} and
\eqref{eq:EstimateIntercept_Eq2} gives
$(\hat{\beta}_0, \hat{\gamma})$, which is then plugged into
\eqref{eq:system_equation_intercept} to compute $\hat{\alpha}$ and
$\hat{\sigma}$.  Eqn. \eqref{eq:ciemp} then provides adjusted
$(1-\alpha)$ confidence interval for a single coefficient $\beta_j$.

Finally, we evaluate the empirical coverage of \eqref{eq:ciemp}.
As in Section \ref{sec:empirical}, we compute the coverage proportion
of a single non-null variable (Table
\ref{tab:cov_prop_nonnull_intercept}) and across all of the variables
(Table \ref{tab:bulk_cov_intercept}). We also study the performance of
the LRT using estimated parameters (Table
\ref{tab:pval_lrt_intercept}). Empirical coverage is accurate since it
is within three standard deviations from the nominal value. The coverage
across all of the variables is slightly smaller than nominal, but the
relative error is within 1\%. The adjusted p-values for the LRT also
control the type I error. We can also see that the results obtained by
using the estimated parameters compare favorably to those obtained
using the theoretical parameters.

\section{Is this all real?}
We have seen that in logistic models with Gaussian covariates of
moderately high dimensions, (a) the MLE overestimates the true effect
magnitudes, (b) the classical Fisher information formula
underestimates the true variability of the ML coefficients, and (c)
classical ML based null p-values are far from uniform. We introduced a
new maximum likelihood theory, which accurately amends all of these
issues and demonstrated empirical accuracy on non-Gaussian
light-tailed covariate distributions. We claim that the issues with ML
theory apply to a broader class of covariate distributions; in fact,
we expect to see similar {\em qualitative} phenomena in real datasets.

Consider the wine quality data \cite{simultech15}, which contains 4898
white wine samples from northern Portugal. The dataset consists of 11
numerical variables from physico-chemical tests measuring various
characteristics of the wine, such as density, pH and volatile acidity,
while the response records a wine quality score that takes on values
in $\{0,\hdots,10\}$. We define a binary response by thresholding the
scores, so that a wine receives a label $y=0$, if the corresponding
score is below $6$, and a label $y=1$, otherwise. We log-transform two
of the explanatory variables as to make their distribution more
symmetrical and concentrated.  We also center the variables so that
each has mean zero.

We explore the behavior of the classical logistic MLE for the variable
``volatile acidity'' ($\mathrm{va}$) at a grid of values of the
problem dimension $\kappa \in K$.  For each $\kappa \in K$, we
construct $B=100$ subsamples containing $n=p/\kappa$ observations
and calculate the MLE $\hat{\beta}_{\mathrm{va}}$ from
  each subsample. Figure \ref{fig:real} shows the boxplots of these
  estimated coefficients. Although the ground truth is unknown, the red dashed line plots the
  MLE $\hat{\beta}_{\mathrm{va}} = -6.18$ calculated over all 4898
  observations so that it is an accurate estimate of the corresponding
  parameter.  Noticeably, the ML coefficients move further from the
  red line, as the dimensionality factor increases, exhibiting a
  strong bias. For instance, when $\kappa$ is in $ \{0.10,0.18,0.26\} $, the
  median MLE is repectively equal to $\{-7.48, -9.43,-10.58\}$; that is, $\{1.21,1.53,1.71 \}$  times the
  value of the MLE (in magnitude) from the full data. 
These observations support our hypothesis that, irrespective of the
covariate distribution, the MLE increasingly overestimates effect
magnitudes in high dimensions.

\begin{figure}
    \centering
    \includegraphics[scale=0.16,keepaspectratio]{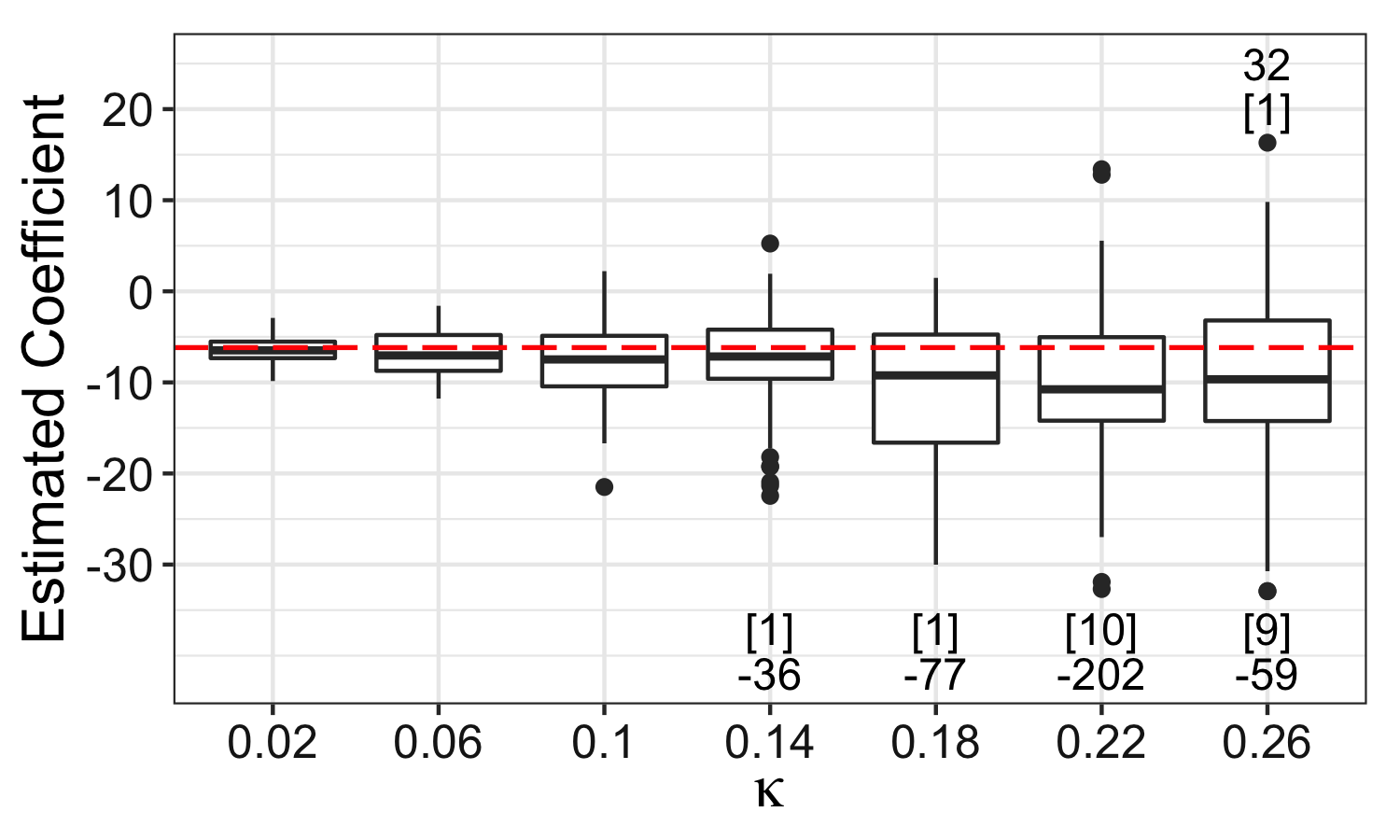}
    \caption{Estimated coefficient
      $\hat{\beta}_\mathrm{va}$ of the variable
      ``volatile acidity'', obtained from $B=100$ sub-samples of size
      $p/\kappa$. The red dashed line shows the MLE using all the
      observations. The numbers of outliers outside of range are those
      between squared brackets. The minimum/maximum value of these
      outliers is given by the accompanying integer.}
      \label{fig:real}
\end{figure}

\begin{figure*}
\centering
    \begin{subfigure}[t]{0.45\textwidth}
        \centering
        \includegraphics[width=1\textwidth]{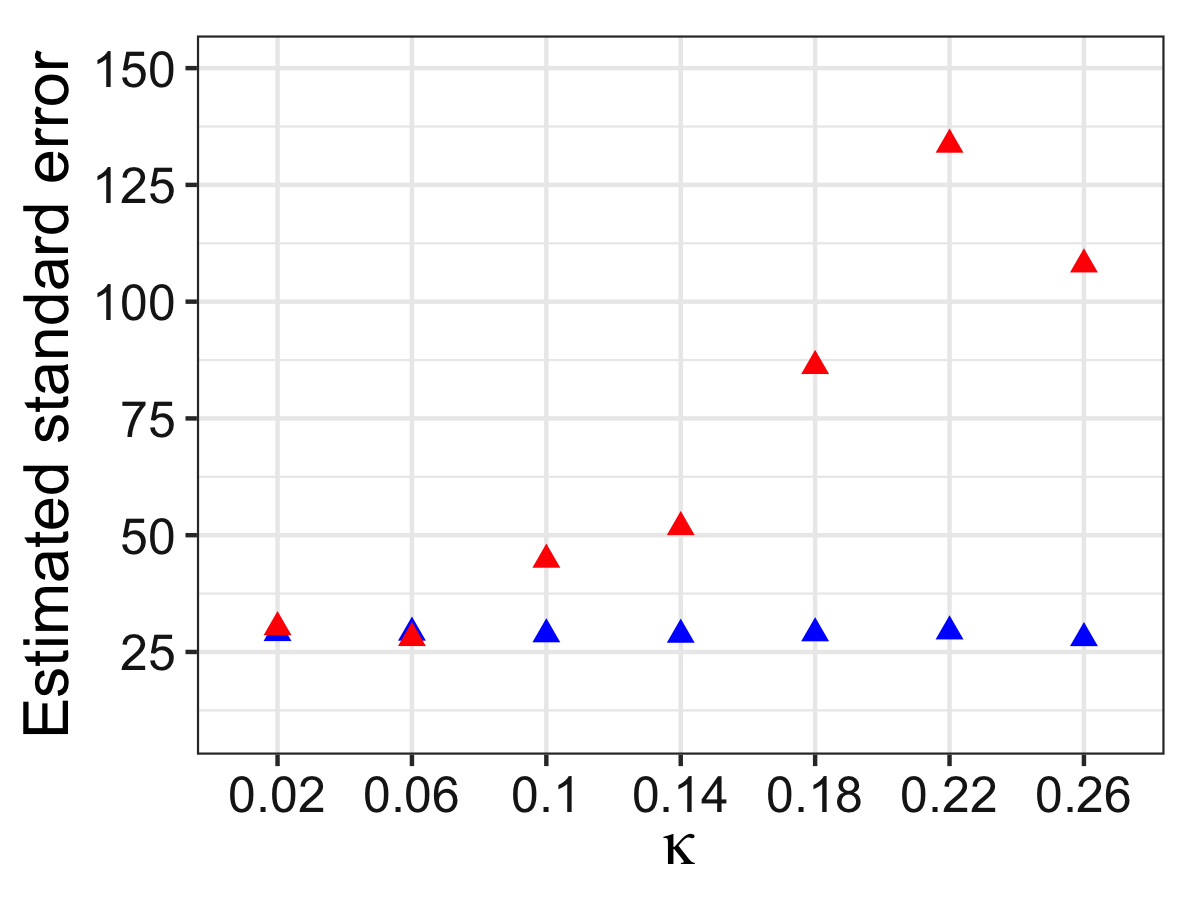}
        \caption{}
        \label{fig:sd_classfish}
    \end{subfigure}
    \hspace*{0em}
    \begin{subfigure}[t]{0.45\textwidth}
        \centering
        \includegraphics[width=1\textwidth]{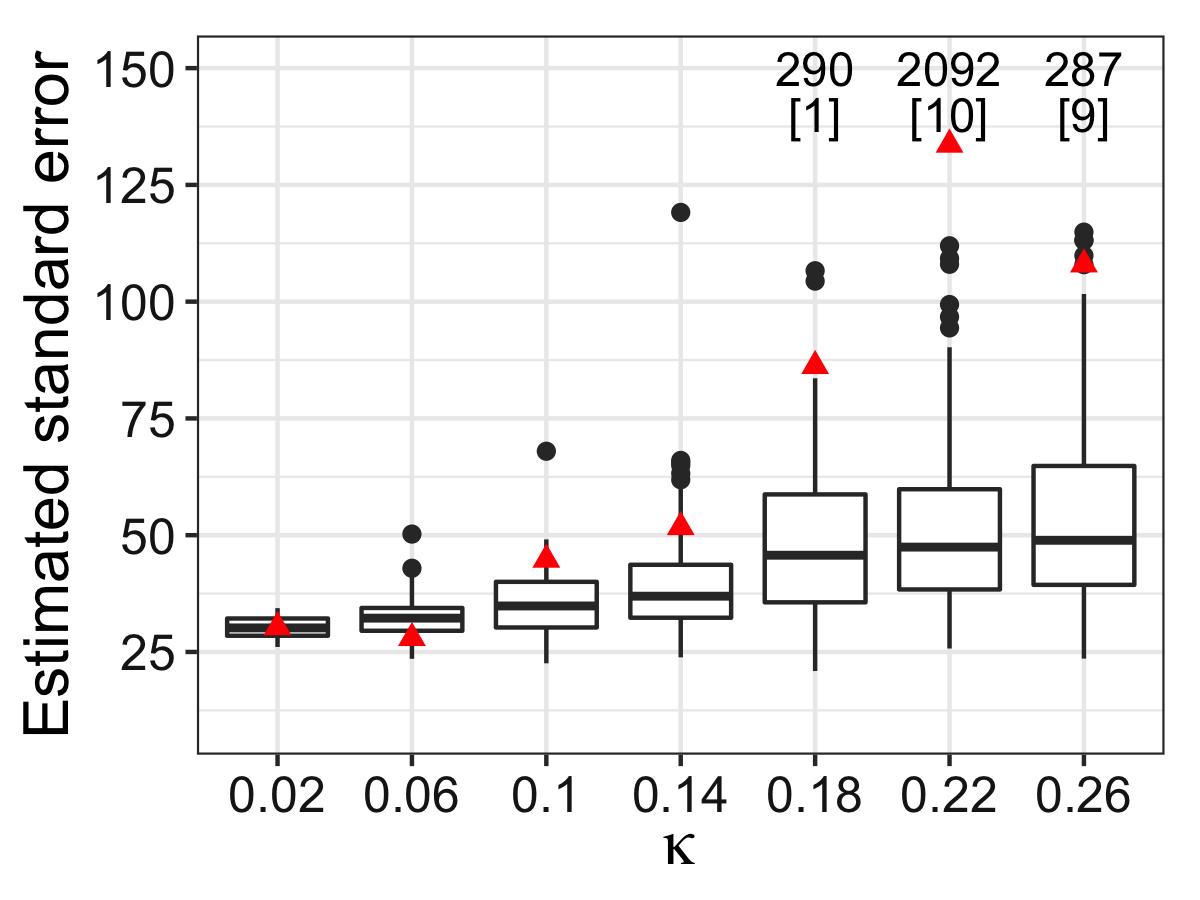}
        \caption{}
        \label{fig:sd_R}
    \end{subfigure}
    \hspace*{0em}
    \caption{Comparison of standard errors adjusted for the
        sample size. Throughout, the red triangles represent an
      estimate of the standard deviation (sd) of the MLE for
      ``volatile acidity'', obtained by creating folds of sample size
      $p/\kappa$ and computing the sd across these folds. To evidence
      the bias, the estimated standard deviations (blue) are multiplied by
      the root sample size (see text).  (a) Estimated sd given by the
      inverse Fisher information (adjusted for
      sample size) averaged over $B = 100$ subsamples
      for each value of $\kappa$. (b) Standard error (adjusted for
      sample size) from \texttt{R} calculated in each subsample. (The meaning
      of the numbers in between square brackets and their accompanying
      integer is as in Figure \ref{fig:real}.)}
    \label{fig:real-sds}
\end{figure*}

Next, Figure \ref{fig:real-sds} compares the standard deviation (sd)
in high dimensions with the corresponding prediction from the
classical Fisher information formula.\footnote{The Fisher information
  here is given by
  $\mathcal{I}(\mbf{\beta}) =\mathbb{E}[\mbf{X}^\top
  \mbf{D}(\mbf{\beta}) \mbf{X}]$, where $\mbf{D}(\mbf{\beta})$ is a
  diagonal matrix with the $i$-th diagonal entry given by
  $\rho''(\mbf{x}_i^\top\mbf{\beta}) =
  e^{\mbf{x}_i^\top\mbf{\beta}}/(1+ e^{\mbf{x}_i^\top
    \mbf{\beta}})^2$.}  Theorem \ref{thm:finitemle} states that when
$n$ and $p$ are both large and the covariates follow a multivariate
Gaussian distribution, $\hat{\beta}_j$ approximately obeys
 \begin{equation}
\label{eq:scaling}
\hat{\beta}_j = \alpha_\star(\kappa)\beta_j + \sigma_\star(\kappa)
Z/\sqrt{n},
\end{equation}
where $Z\sim\mathcal{N}(0,1)$. If we reduce $n$ by a factor of 2, the
standard deviation should increase by a factor of
$\sigma_{\star}(2\kappa)/\sigma_{\star}(\kappa) \times\sqrt{2}$. Thus,
in order to evidence the interesting contribution, namely, the factor
of $\sigma_{\star}(2\kappa)/\sigma_{\star}(\kappa) $, we plot
$\sqrt{n}\times \mathrm{sd}(\hat{\beta}_\mathrm{va})$, where $n$ is
the sample size used to calculate our estimate.

With the sample size adjustment, we see in Figure
\ref{fig:sd_classfish} that the variance of the ML coefficient is much
higher than the corresponding classical value, and that the mismatch
increases as the problem dimension increases. Thus, we see once more a
``variance inflation'' phenomenon similar to that observed for
Gaussian covariates (see also, \cite{elkaroui13, donoho16}). To be
complete, we here approximate/estimate the (inverse) Fisher
information as follows: for each $\kappa \in K$, we form
$\hat{\mathcal{I}}(\mbf{\beta})=\frac{1}{B} \sum_{j=1}^B \mbf{X}_j'
\mbf{D}_{\mbf{\beta}} \mbf{X}_j$, where $\mbf{X}_j$ is the covariate
matrix from the $j$-th subsample, and for $\mbf{\beta}$, we plug in
the MLE from the full data.

Standard errors obtained from software packages are different from
those shown in Figure \ref{fig:sd_classfish}, since these typically
use the maximum likelihood estimate $\hat{\mbf{\beta}}$ from the data
set at hand as a plug-in for $\mbf{\beta}$, and in addition, do not
take expectation over the randomness of the covariates. However, since
these estimates are widely used in practice, it is of interest to
contrast them with the true standard deviations.  Figure
\ref{fig:sd_R} presents a boxplot of standard errors of $\hat{\beta}_\mathrm{va}$
(adjusted for sample size) as obtained from \texttt{R}.  Observe that
for large values of $\kappa$, these also severely underestimate the
true variability.

\section{Discussion}

This paper establishes a maximum likelihood theory for
high-dimensional logistic models with arbitrarily correlated Gaussian
covariates. In particular, we establish a stochastic representation
for the MLE that holds for finite sample sizes. This in turn yields a
precise characterization of the finite-dimensional marginals of the
MLE, as well as the average behavior of its coordinates. Our theory
relies on the unknown signal strength parameter $\gamma$, which can be
accurately estimated by the \emph{ProbeFrontier} method. This
provides a valid procedure for constructing p-values and confidence
intervals for any finite collection of coordinates.  Furthermore, we
observe that our procedure produces reliable results for moderate
sample sizes, even in the absence of Gaussianity---in particular, when
the covariates are light-tailed.

We conclude with a few directions of future research---it would be of
interest to understand (a) the class of covariate distributions for
which our theory, or a simple modification thereof, continues to
apply, (b) the class of generalized linear models for which analogous
results hold, and finally, (c) the robustness of our proposed
procedure to model misspecifications.


\section*{Acknowledgements}
E.~C.~was supported by the National Science Foundation via DMS 1712800
and via the Stanford Data Science Collaboratory OAC 1934578, and by a 
generous gift from TwoSigma. P.S.~was supported by the Center for Research on Computation and Society, Harvard John A.~Paulson School of Engineering and Applied Sciences. Q.~Z.~would like to thank Stephen Bates
for helpful comments about an early version of this paper.
%
%

\begin{supplement}
\textbf{Auxiliary results and proofs}.
This supplementary material contains the proofs of Theorem \ref{thm:bulk}--\ref{thm:general_testfun} and additional simulation results for Section \ref{sec:intercept_method}. 
\end{supplement}

\begin{supplement}
\textbf{Code to reproduce results in the article.} This supplementary material contains code to reproduce simulation results in this article, also available at \url{https://github.com/zq00/logisticMLE}.
\end{supplement}

%
%
%
%
%
%
%




\end{document}